\documentclass[mathpazo]{csiam-am}
\renewcommand\thisnumber{x}
\renewcommand\thisyear {20xx}
\renewcommand\thismonth{xx}
\renewcommand\thisvolume{x}
\usepackage{booktabs}
\usepackage{amsfonts}
\usepackage{graphicx}
\usepackage{epstopdf}
\usepackage{mathrsfs}
\usepackage{algorithm}
\usepackage{algpseudocode,float}
\usepackage{mathdots}
\usepackage{multirow}
\usepackage[shortlabels]{enumitem}
\usepackage{comment}
\usepackage{rotating}
\usepackage{xcolor}
\usepackage{titlecaps}
\usepackage{caption}
\usepackage{subcaption}
\usepackage{xargs}
\usepackage{bm}

\usepackage{amsopn}

\DeclareMathOperator*{\esssup}{ess\,sup}
\DeclareMathOperator*{\essinf}{ess\,inf}

\newcommand{\innerprod}[2]{\left\langle{#1},#2\right\rangle}
\newcommand{\matcomplexify}[1]{\mathcal{M}{(#1)}}

\newcommand{\operationcount}[1]{\mathcal{O}{(#1)}}
\newcommand{\vectocirc}[1]{\mathcal{C}(#1)}

\newcommand{\vectocpoly}[2]{\mathcal{F}\left(#1,#2\right)}

\newcommand{\strangfcoloftoep}[1]{{\bf v}_S\left(#1\right)}
\newcommand{\bfrefnumber}[1]{${\bf (#1)}$}
\newcommand{\matcmplxexpress}[1]{\left[\begin{array}[c]{cc}
		\phi_1(#1)&-\phi_2(#1)\\
		&\\
		\overline{\phi_2(#1)}&\overline{\phi_1(#1)}
	\end{array}\right]}
\newcommand{\veccomplexify}[1]{\mathcal{V}{\left(#1\right)}}
\newcommand{\invveccomplexify}[1]{\mathcal{V}^{-1}{(#1)}}
\newcommand{\veccmplxexpress}[1]{\left[\begin{array}[c]{c}
		\phi_1(#1)\\
		~\\
		\overline{\phi_2(#1)}
	\end{array}\right]}

\newcommand{\uniti}{{{\tt i}}}
\newcommand{\unitj}{{{\tt j}}}
\newcommand{\unitk}{{{\tt k}}}
\newcommand{\unitp}{{{\tt p}}}
\newcommand{\unitq}{{{\tt q}}}
\newcommand{\unitr}{{{\tt r}}}
\newcommand{\transpose}{{{\rm T}}}
\usepackage{titlecaps}

\begin{document}
\title{Hermitian Quaternion Toeplitz Matrices by Quaternion-valued Generating 
	Functions}


\author[Xue-Lei Lin, Michael K. Ng and Junjun Pan]{Xue-Lei Lin\affil{1}\ and Michael K. Ng\affil{2}\comma\corrauth\ and Junjun Pan\affil{2}}
\address{\affilnum{1}\ School of Science, Harbin Institute of Technology, Shenzhen 518055, China\\
\affilnum{2}\ Department of Mathematics, Hong Kong Baptist University, Hong Kong}

%
%
\emails{{\tt linxuelei@hit.edu.cn} (Xue-Lei Lin), {\tt michael-ng@hkbu.edu.hk} ( Michael K. Ng),  {\tt junjpan@hkbu.edu.hk} (Junjun Pan)}

\begin{abstract}
In this paper, we study Hermitian quaternion Toeplitz matrices generated by 
quaternion-valued functions. We show that such generating function 
must be the sum 
of a real-valued function and an odd function with imaginary component. 
This setting is different from the case of Hermitian complex Toeplitz matrices generated by real-valued functions only. 
By using of 2-by-2 block complex representation of quaternion 
matrices, we give a quaternion version of 
Grenander-Szeg\"{o} theorem stating the distribution of eigenvalues of Hermitian quaternion Toeplitz matrices in terms of its generating function. 
As an application, we investigate Strang's circulant preconditioners for 
Hermitian quaternion Toeplitz linear systems arising from quaternion signal processing.
We show that Strang's circulant preconditioners can be diagonalized by 
discrete quaternion Fourier transform matrices 
whereas general quaternion circulant matrices cannot be diagonalized by them. 
Also we verify the theoretical and numerical convergence results of 
Strang's circulant preconditioned conjugate gradient method for 
solving Hermitian quaternion Toeplitz systems.
\end{abstract}


\maketitle


	\section{Introduction}\label{sec:introduction}
A quaternion number is a number of form $a_0+a_1\uniti+a_2\unitj+a_3\unitk$, where $a_s$'s are all real numbers; $\uniti,\unitj,\unitk$ are imaginary units such that 
$$		
\uniti^2=\unitj^2=\unitk^2=-1,\quad\uniti\unitj=-\unitj\uniti=\unitk,\quad \unitj\unitk=-\unitk\unitj=\uniti,\quad \unitk\uniti=-\uniti\unitk=\unitj;
$$
{\color{black}We call a quaternion as pure quaternion if $a_0=0$.} 
In this work, we 
are interested in solving Hermitian quaternion Toeplitz systems.
\begin{equation}\label{hpdqtsystem}
	T_n \mathbf{u}=\mathbf{b},
\end{equation}
where $b$ is a given quaternion vector of size $n\times 1$; {\color{black}$u$ is the unknown to be solved;} $T_n$ is an $n\times n$ Hermitian quaternion Toeplitz matrix of the following form
\begin{equation*}
	T_n=\left[
	\begin{array}
		[c]{ccccc}
		t_0  &  t_{-1} &\ldots   & t_{2-n}   & t_{1-n} \\
		t_1  &  t_0 & t_{-1}    &\ddots & t_{2-n} \\
		\vdots&\ddots &\ddots&\ddots&\vdots\\
		t_{n-2}  & \ddots& t_1 &  t_0  & t_{-1} \\
		t_{n-1}  &  t_{n-2} & \ldots &  t_1  &  t_{0} 
	\end{array}
	\right],
\end{equation*}
with $t_s$'s are quaternion numbers; $t_{-s}$ is conjugate of $t_s$ (see Section 2 for the definition) for each $s>0$; 
$t_0$ is a real number. 
Quaternion Toeplitz systems have attracted a growing attention in recently years, thanks to its various applications including signal and image processing \cite{took2008quaternion,miron2023quaternions,jia2021structure,li2023structure,donatelli2006improved,di2005superoptimal,donatelli2007boundary}.

The main aim of this paper is to 
show that the eigenvalues of Hermitian 
quaternion Toeplitz matrices can be characterized by their quaternion-valued 
generating function. 
This refers to a quaternion-version of Grenander-Szeg\"{o} theorem.
We note that such generating function must be the sum 
of a real-valued function and an odd function with imaginary component. 
This setting is different from the case of Hermitian complex Toeplitz matrices generated by real-valued functions only. 
As an application, 
we investigate circulant preconditioned conjugate gradient 
methods for solving Hermitian quaternion Toeplitz systems arising from quaternion signal processing. 
Here we employ Strang's circulant preconditioners (as examples) 
and 
show that the eigenvalues of Strang's circulant matrices are 
computed by the values evaluated at the partial sum of generating functions. 
We remark that a quaternion circulant matrix in general is not diagonalized by discrete Fourier transform matrix.
Also we verify when the circulant preconditioned conjugate gradient method is applied to solving such Hermitian quaternion Toeplitz systems, its 
superlinear convergence can be achieved. 

The outline of this paper is given as follows.  In Section 2, we study the spectra
spectra of Hermitian quaternion Toeplitz matrices
constructed by quaternion-valued generating function.
In Section 3, we study the spectra of circulant preconditioned Toeplitz matrices.
Numerical results for Hermitian quaternion Toeplitz matrices arising quaternion
signal processing are reported in Section 4. Finally, some
concluding remarks are given in Section 5.


\subsection{Preliminaries}

The real number field, the complex number field and the quaternion algebra are denoted by $\mathbb{R}$, $\mathbb{C}$ and $\mathbb{Q}$,  respectively. 
The set of all positive real numbers is denoted by $\mathbb{R}^{+}$.	
The spaces of all $m\times n$ real matrices, complex matrices and quaternion matrices are denoted by $\mathbb{R}^{m\times n}$, $\mathbb{C}^{m\times n}$ and $\mathbb{Q}^{m\times n}$,  respectively.	
If there is no ambiguity, we treat  $\mathbb{R}^{1\times 1}$, $\mathbb{C}^{1\times 1}$ and $\mathbb{Q}^{1\times 1}$ as 	$\mathbb{R}$, $\mathbb{C}$ and $\mathbb{Q}$, respectively.
The set of all positive integers and the of all nonnegative integers are denoted by $\mathbb{N}^{+}$ and $\mathbb{N}$, respectively.

{\color{black}The dot product of two quaternion numbers $x=x_0+x_1\uniti+x_2\unitj+x_3\unitk$ and $y=y_0+y_1\uniti+y_2\unitj+y_3\unitk$ 
	with $x_s,y_s\in\mathbb{R}$ for $s=0,1,2,3$, is defined as
	\begin{equation*}
		x \cdot y:=\sum\limits_{s=0}^{3}x_sy_s.
	\end{equation*}
	The modulus of $x$ is defined by	 
	$|x|:=\sqrt{x\cdot x}=\sqrt{|x_0|^2+|x_1|^2+|x_2|^2+|x_3|^2}$.
	We say $x$ is a unit if and only if $|x|=1$.	Two pure quaternions $x,y$ are said to be orthogonal  if and only if $x\cdot y=0$.
}

\begin{definition}\label{orthogonalunitsdef}
	We call a triple of pure quaternions $(\unitp,\unitq,\unitr)$ orthonormal if and only if  $\unitp,\unitq,\unitr$ are all units and any two of them are orthogonal.
\end{definition}

{\color{black}It can be seen that  $(\uniti,\unitj,\unitk)$ is orthonormal. 
	Any orthonormal triple of pure quaternions $(\unitp,\unitq,\unitr)$ can be treated as a three-axis imaginary system like $(\uniti,\unitj,\unitk)$, see Property 1 in \cite{pan2024block}.}

For any $A=\tilde{A}_0+\tilde{A}_1\unitp+\tilde{A}_2\unitq+\tilde{A}_3\unitr\in\mathbb{Q}^{m_1\times m_2}$ with $\tilde{A}_s\in\mathbb{R}^{m_1\times m_2}~(s=0,1,2,3)$, its transpose $A^\transpose$ and its conjugate $\bar{A}$ are defined as 
\begin{align*}
	A^{\transpose}:=\tilde{A}_0^\transpose+\tilde{A}_1^\transpose\unitp+\tilde{A}_2^\transpose\unitq+\tilde{A}_3^\transpose\unitr,\quad \bar{A}:=\tilde{A}_0-\tilde{A}_1\unitp-\tilde{A}_2\unitq-\tilde{A}_3\unitr;
\end{align*}
For any $A\in\mathbb{Q}^{m_1\times m_2}$, the conjugate transpose of $A$, denoted by $A^{*}$,  is  $A^{*}:=\bar{A}^{\rm T}$.
$A$ is said to be Hermitian if and only if $A=A^{*}$. 

\textcolor{black}{For any ${\bf x}=(x_1,x_2,...,x_m)^{\rm T}\in\mathbb{Q}^{m\times 1}$, its vector 2-norm is defined as $||{\bf x}||_2=\sqrt{\sum\limits_{i=1}^{m}|x_i|^2}$. Clearly, $||{\bf x}||_2=\sqrt{{\bf x}^{*}{\bf x}}$ for any quaternion-valued vector ${\bf x}$.}

Denote
\begin{equation*}
	\mathbb{Q}_{(0,1)}:=\{x+y\unitp|x,y\in\mathbb{R}\},\quad \mathbb{Q}_{(2,3)}:=\{x\unitq+y\unitr|x,y\in\mathbb{R}\}.
\end{equation*}
{\color{black} The set of all $m\times n$ matrices whose entries belonging to $\mathbb{Q}_{(0,1)}$ (or $\mathbb{Q}_{(2,3)}$) is denoted as  $ \mathbb{Q}_{(0,1)}^{m\times n}$ (or $\mathbb{Q}_{(2,3)}^{m\times n}$),}  
We remark that $\mathbb{Q}_{(0,1)}$ is a field isomorphic to the complex field. Hence, results on numbers in $\mathbb{C}$ (or matrices in $\mathbb{C}^{m\times n}$) can be applied to the numbers in $\mathbb{Q}_{(0,1)}$ (or matrices in $ \mathbb{Q}_{(0,1)}^{m\times n}$) under isomorphism. 
	For a square matrix $A\in\mathbb{K}^{m\times m}$ 
		($\mathbb{K}=\mathbb{Q}$ or $\mathbb{Q}_{(0,1)}$),
	we denote
	\begin{equation*}
		\sigma(A):=\{\lambda\in\mathbb{K}|A{\mathbf{z}}=\mathbf{ z}\lambda{\rm~for~some~nonzero~vector~}\mathbf{z}\in\mathbb{K}^{m\times 1}\}.
	\end{equation*}
	Here we refer $(\lambda,\mathbf{z})$ be the right eigenvalue and corresponding eigenvector of $A$, and
	$\sigma(A)$ be the set of all right eigenvalues of $A$. 
	
	\begin{definition}\label{qhpddef}
		A Hermitian matrix $A\in\mathbb{Q}^{m\times m}$ is said to be Hermitian positive definite (HPD) if and only if ${\bf x}^{*}A{\bf x}>0$ holds for any nonzero ${\bf x}$. 
	\end{definition}
	
	\begin{lemma}\label{hpdqmatproplm}\textnormal{(see, e.g., \cite{opfer2005conjugate})}
		Let $A\in\mathbb{Q}^{n\times n}$ be Hermitian matrix. Then, we have
		\begin{description}
			\item[(i)] $\mathbf{x}^{*}A\mathbf{x}\in\mathbb{R}$ hold for all $\mathbf{x}\in\mathbb{Q}^{n\times 1}$;
			\item[(ii)]  $\sigma(A)\subset\mathbb{R}$ and the number of distinct right eigenvalues of 
			$A$ is at most $n$.
			If in addition, $A$ is HPD, then $\sigma(A)\subset\mathbb{R}^{+}$. 
		\end{description}
	\end{lemma}

	
	For any $A=A_0+A_1\unitp+A_2\unitq+A_3\unitr\in\mathbb{Q}^{m_1\times m_2}$ with $A_s\in\mathbb{R}^{m_1\times m_2}$ for $s=0,1,2,3$, $A$ can be equivalently expressed as
	\begin{equation*}
		A=(A_0+A_1\unitp)+(A_2+A_3\unitp)\unitq,
	\end{equation*}			
	and such representation is unique. 
	
	\begin{definition}
		For any positive integers, $m_1$ and $m_2$, for any $A=A_0+A_1\unitp+A_2\unitq+A_3\unitr\in\mathbb{Q}^{m_1\times m_2}$ with $A_s\in\mathbb{R}^{m_1\times m_2}$ for $s=0,1,2,3$, define $\phi_l(A)$ $(l=1,2)$ 
		as follows: 
		\begin{equation}\label{phi1phi2def}
			\phi_1(A):=A_0+A_1\unitp,\quad \phi_2(A):=A_2+A_3\unitp.
		\end{equation}
	\end{definition}
	From the above definition, $A=\phi_1(A)+\phi_2(A)\unitq$ holds for any positive integers, $m_1$ and $m_2$ and any $A\in\mathbb{Q}^{m_1\times m_2}$. If $A$ is a scalar,  $\phi_1(A)$ and $\phi_2(A)$ can be defined in a similar way to \eqref{phi1phi2def}. 
	
	With $\phi_1(\cdot)$ and $\phi_2(\cdot)$, we 
	define two mappings $\mathcal{M}(\cdot)$ and $\mathcal{V}(\cdot)$ as follows:
	\begin{align*}
		&\mathcal{M}(\cdot):A\in\mathbb{Q}^{n\times n}\mapsto\matcomplexify{A}:=\matcmplxexpress{A}\in\mathbb{Q}_{(0,1)}^{2n\times 2n};\\
		&\mathcal{V}(\cdot):\mathbf{x}\in\mathbb{Q}^{n\times 1}\mapsto\veccomplexify{\mathbf{x}}:=\veccmplxexpress{\mathbf{x}}\in\mathbb{Q}_{(0,1)}^{2n\times 1}.
	\end{align*}
		Note that if $A \in \mathbb{Q}^{n \times n}$ is Hermitian, then 
		$\mathbf{x}^* A \mathbf{x} = \mathcal{V}(\mathbf{x})^* \mathcal{M}(A) \mathcal{V}(\mathbf{x})$ and 
		$\sigma(A) =\sigma( \mathcal{M}(A))$.
	
	\section{Quaternion Toeplitz Matrices}
	\subsection{Generating Functions}
	
		For a quaternion-valued function  $f$, we define two $\mathbb{Q}_{(0,1)}$-valued functions $\Phi_1[f]$ and $\Phi_2[f]$ as follows:
	\begin{equation*}
		\Phi_1[f](x):=\phi_1(f(x)),\quad \Phi_2[f](x):=\phi_2(f(x)).
	\end{equation*}	 
	With the notations introduced above, we define a pair of functions for
	generating  
	a sequence of quaternion Hermitian Toeplitz matrices.
	
	\begin{definition}\label{quatgfuncdef}
		A quaternion-valued function $f\in L^1([-\pi,\pi])$ is called a generating function for a sequence of quaternion 
		Hermitian Toeplitz matrices $\{T_n\}_{n\in\mathbb{N}^{+}}$ if 
		$$
		t_s =\frac{1}{2\pi}\int_{-\pi}^{\pi}f(x)\exp(-\unitp sx)dx, 
		\quad \bar{t}_s =\frac{1}{2\pi}\int_{-\pi}^{\pi}f(x)\exp(\unitp sx)dx,~s=0,1,..., 
		$$
		where  $\exp(\unitp x):=\cos(x)+\unitp\sin(x),~ x\in\mathbb{R}$.
	\end{definition}
	
	\begin{lemma}\label{gfuncstructurelm}
		Let $f$ be a generating function defined in Definition \ref{quatgfuncdef} for the quaternion Toeplitz Hermitian matrices $\{ T_n \}$. 
		Then,
		\begin{description}
			\item[(i)] $\Phi_1[f]\in L^{1}([-\pi,\pi])$ and $\Phi_1[f](x)\in\mathbb{R}$ a.e. $x\in[-\pi,\pi]$; 
			\item[(ii)] $\Phi_2[f]\in L^{1}([-\pi,\pi])$ and $\Phi_2[f](-x)=-\Phi_2[f](x)$  a.e. $x\in[-\pi,\pi]$.
		\end{description}
	\end{lemma}
	
	\begin{proof}
		Since $f\in L^1([-\pi,\pi])$ and $|\phi_l(f)(x)|\leq |f(x)|$ for $l=0,1$, it is clear that $\phi_l(f)\in L^1([-\pi,\pi]) $ for $l=0,1$.
		From  Definition \ref{quatgfuncdef}, we see that
		\begin{align}
			t_s	=&\frac{1}{2\pi}\int_{-\pi}^{\pi}f(x)\exp(-\unitp sx)dx\notag\\
			& \hspace{-4mm} =\frac{1}{2\pi}\int_{-\pi}^{\pi}[\Phi_1[f](x)+\Phi_2[f](x)\unitq]\exp(-\unitp sx)dx\notag\\
			& \hspace{-4mm} =\frac{1}{2\pi}\int_{-\pi}^{\pi}\Phi_1[f](x)\exp(-\unitp sx)dx+\left[\frac{1}{2\pi}\int_{-\pi}^{\pi}\Phi_2[f](x)\exp(\unitp sx)dx\right]\unitq, \ s=0,\pm 1,\pm 2,..., \label{tsdecompos}
		\end{align}
		which implies that
		\begin{equation*}
			\phi_1(t_s)=\frac{1}{2\pi}\int_{-\pi}^{\pi}\Phi_1[f](x)\exp(-\unitp sx)dx,\quad s=0,\pm 1,\pm 2,...
		\end{equation*}
		Hence, 
		\begin{equation*}
			\overline{\phi_1(t_s)}=\frac{1}{2\pi}\int_{-\pi}^{\pi}\overline{\Phi_1[f](x)}\exp(\unitp sx)dx.
		\end{equation*}
		On the other hand, Definition \ref{quatgfuncdef} and \eqref{tsdecompos} imply that
		\begin{align*}
			\phi_1(\bar{t}_s)&=\phi_1\left(\frac{1}{2\pi}\int_{-\pi}^{\pi}f(x)\exp(-\unitp (-s)x)dx\right)\\
			&=\frac{1}{2\pi}\int_{-\pi}^{\pi}\Phi_1[f](x)\exp(\unitp sx)dx,\quad \textcolor{black}{s=0,\pm 1,\pm 2,....}
		\end{align*}
		By the simple fact that $ \phi_1(\bar{t}_s)= \overline{\phi_1(t_s)}$, we have 
		\begin{equation*}
			\frac{1}{2\pi}\int_{-\pi}^{\pi}\overline{\Phi_1[f](x)}\exp(\unitp sx)dx=\frac{1}{2\pi}\int_{-\pi}^{\pi}\Phi_1[f](x)\exp(\unitp sx)dx,\quad s=0,\pm 1,\pm 2,...
		\end{equation*}
		or equivalently
		\begin{equation*}
			\frac{1}{2\pi}\int_{-\pi}^{\pi}\left[\overline{\Phi_1[f](x)}-\Phi_1[f](x)\right]\exp(\unitp sx)dx=0,\quad s=0,\pm 1,\pm 2,...
		\end{equation*}
		Note that the set of polynomials $$P([-\pi,\pi]):=\left\{\sum\limits_{s=-\ell}^{\ell}a_s\exp(\unitp sx)\Big|a_s\in\mathbb{Q}_{(0,1)},s=
		-\ell,...,-1,0,1,...,\ell,\quad \ell\in\mathbb{N}\right\}$$ is dense in $C([-\pi,\pi])$, where $C([-\pi,\pi])$ denotes the set of all $\mathbb{Q}_{(0,1)}$-valued continuous functions defined on $[-\pi,\pi]$. Hence,
		\begin{equation*}
			\int_{-\pi}^{\pi}\left[\overline{\Phi_1[f](x)}-\Phi_1[f](x)\right]h(x)dx=0,\quad \forall h\in C([-\pi,\pi]).
		\end{equation*}
		That means $\overline{\Phi_1[f](x)}-\Phi_1[f](x)=0$ ~ a.e. $x\in[-\pi,\pi]$. In other words, $\Phi_1[f](x)\in\mathbb{R}$ a.e. $x\in[-\pi,\pi]$. The proof of \bfrefnumber{i} is complete.
		
		On the other hand, \textcolor{black}{according to  \eqref{tsdecompos} and Definition \ref{quatgfuncdef}}, we know that
			\begin{align*}
				\phi_2(t_s)&=\frac{1}{2\pi}\int_{-\pi}^{\pi}\Phi_2[f](x)\exp(\unitp sx)dx,\quad s=0,1,...\\
				\phi_2(\bar{t}_s)&=\phi_2\left(\frac{1}{2\pi}\int_{-\pi}^{\pi}f(x)\exp(\textcolor{black}{\unitp sx})dx\right)\\
				&=\frac{1}{2\pi}\int_{-\pi}^{\pi}\Phi_2[f](x)\exp(-\unitp sx)dx= -\phi_2(t_s),\quad s=0,1,...\\
			\end{align*}
		Therefore, we have 
		\begin{equation*}
			\int_{-\pi}^{\pi}-\Phi_2[f](x)\exp(\unitp sx)dx=\int_{-\pi}^{\pi}\Phi_2[f](x)\exp(-\unitp sx)dx,\quad s\in\mathbb{N}.
		\end{equation*}
		In other words,
		\begin{equation*}
			\int_{-\pi}^{\pi}\Phi_2[f](x)[\exp(-\unitp sx)+\exp(\unitp sx)]dx=0,\quad s\in\mathbb{N},
		\end{equation*}
		based on which one can show that $\Phi_2[f](-x)=-\Phi_2[f](x)$ a.e. $x\in[-\pi,\pi]$. The proof is complete.
	\end{proof}
	
	According to Lemma 2.2, we remark that 
	a sequence of Hermitian quaternion Toeplitz matrices is 
	generated by a quaternion-valued function which is the sum 
	of a real-valued function and an odd function with imaginary component. 
	This setting is different from the case of Hermitian complex Toeplitz matrices generated by real-valued functions only. 
	
	For any vector $\mathbf{y}\in\mathbb{Q}^{n\times 1}$, denote 
	\begin{equation*}
		\vectocpoly{\mathbf{ y}}{x}:=\sum\limits_{s=1}^{n}\mathbf{y}(s)\exp(\unitp s x),\quad x \in[-\pi,\pi].
	\end{equation*}
	Next we characterize the eigenvalues of quaternion Hermitian Toeplitz matrices. 
	
	\begin{theorem}\label{qhermitmatrayleightogfuncthm}
		Let $f$ be a generating function defined in Definition \ref{quatgfuncdef} for the quaternion Toeplitz Hermitian matrices $\{ T_n \}$. 
		\begin{description}
			\item[(i)] For any $\mathbf{ y}\in\mathbb{Q}^{n\times 1}$, it holds that
			\begin{equation*}
				\mathbf{ y}^{*}T_n\mathbf{ y}=\frac{1}{2\pi}\int_{-\pi}^{\pi}[\veccomplexify{\vectocpoly{\mathbf{ y}}{x}}]^{*}G[f](x)\veccomplexify{\vectocpoly{\mathbf{ y}}{x}}dx,
			\end{equation*}
			where 
			\begin{align*}
				&\textcolor{black}{G[f](x):=\left[\begin{array}[c]{cc}
					\Phi_1[f](x)& \Phi_2[f](x)\\
					&\\
					\overline{\Phi_2[f](x)}&\Phi_1[f](-x)
				\end{array}\right]\in\mathbb{Q}_{(0,1)}^{2\times 2},\quad x\in[-\pi,\pi].}
			\end{align*}
			It is clear that $G[f](x)$ is Hermitian for each $x$.
			\item[(ii)] Let
			\begin{small}
				\begin{align*}
					&\check{f}(x):=\frac{1}{2}\left[\Phi_1[f](x)+\Phi_1[f](-x)-\sqrt{|\Phi_1[f](x)-\Phi_1[f](-x)|^2+4|\Phi_2[f](x)|^2}\right],\\
					&\hat{f}(x):=\frac{1}{2}\left[\Phi_1[f](x)+\Phi_1[f](-x)+\sqrt{|\Phi_1[f](x)-\Phi_1[f](-x)|^2+4|\Phi_2[f](x)|^2}\right].
				\end{align*}
			\end{small}			
			Then, we have 
			\begin{align*}
				&\essinf\limits_{x\in[-\pi,\pi]}\check{f}(x)=\essinf\limits_{x\in[-\pi,\pi]}\lambda_{\min} (G[f](x))\leq \lambda_{\min}(T_n),\\
				& \lambda_{\max}(T_n)\leq \esssup\limits_{x\in[-\pi,\pi]}\lambda_{\max} (G[f](x))=\esssup\limits_{x\in[-\pi,\pi]}\hat{f}(x).
			\end{align*}
			\item[(iii)] If there is no constant function equal to $\check{f}$ $(\hat{f},~{\rm respectively} )$  a.e. on 
			$[-\pi,\pi]$, then
			\begin{equation*}
				\essinf\limits_{x\in[-\pi,\pi]}\check{f}(x)<\lambda_{\min}(T_n)
				\quad {\rm and} \quad \esssup\limits_{x\in[-\pi,\pi]}\hat{f}(x)>\lambda_{\max}(T_n).
			\end{equation*}
		\end{description}
	\end{theorem}
	
	\begin{proof}
		For {\bf (i)}, we first note that
		\begin{equation*}
			\mathbf{y}^{*}T_n\mathbf{y}=\veccomplexify{\mathbf{y}}^{*}\matcomplexify{T_n}\veccomplexify{\mathbf{y}}.
		\end{equation*}
		Next we derive 
		\begin{align}
			\mathbf{ y}^{*}T_n\mathbf{ y} & =\underbrace{[\phi_1(\mathbf{ y})]^{*}\phi_1(T_n)\phi_1
				(\mathbf{ y})}_{:=term \ 1}+\underbrace{\left(\overline{\phi_2(\mathbf{ y})}\right)^{*}\overline{\phi_2(T_n)}\phi_1(\mathbf{ y})}_{:=term \ 2}\notag
			\underbrace{-[\phi_1(\mathbf{ y})]^{*}\phi_2(T_n)\overline{\phi_2(\mathbf{ y})}}_{:=term \ 3}+ \\
			& \quad \underbrace{\left(\overline{\phi_2(\mathbf{ y})}\right)^{*}\overline{\phi_1(T_n)}~\overline{\phi_2(\mathbf{ y})}}_{:=term \ 4}.\label{tnfiledvalueeq1}
		\end{align}
		Here $\phi_1(T_n)$ is a Toeplitz matrix with
		\begin{equation*}
			[\phi_1(T_n)](s,\ell)=\begin{cases}
				\phi_1(t_{s-\ell}),\quad s-\ell\geq 0,\\
				\phi_1(\bar{t}_{\ell-s}),\quad s-\ell<0.
			\end{cases}
		\end{equation*} 
			By Definition \ref{quatgfuncdef}, we have 
			\begin{equation*}
				[\phi_1(T_n)](s,\ell)=\begin{cases}
					\phi_1(t_{s-\ell})=\frac{1}{2\pi}
					{\displaystyle \int_{-\pi}^{\pi} } \Phi_1[f](x)\exp(-\unitp (s-\ell)x)dx,\quad s-\ell\geq 0,\\
					\phi_1(\bar{t}_{\ell-s})=\frac{1}{2\pi}
					{\displaystyle \int_{-\pi}^{\pi} } 
					\Phi_1[f](x)\exp(-\unitp (s-\ell)x)dx,\quad s-\ell<0.
				\end{cases}
			\end{equation*} 
			Then, Lemma \ref{gfuncstructurelm} \bfrefnumber{i} implies that
			\begin{align*}
				term \ 1&=\sum\limits_{s=1}^{n}\sum\limits_{\ell=1}^{n}\overline{[\phi_1(\mathbf{ y})](s)}[\phi_1(T_n)](s,\ell)[\phi_1(\mathbf{ y})](\ell)\\
				&=\sum\limits_{s=1}^{n}\sum\limits_{\ell=1}^{n}\overline{[\phi_1(\mathbf{ y})](s)}\frac{1}{2\pi} {\displaystyle \int_{-\pi}^{\pi} }\Phi_1[f](x)\exp(-\unitp (s-\ell)x)dx[\phi_1(\mathbf{ y})](\ell)\\
				&=\frac{1}{2\pi}
				{\displaystyle \int_{-\pi}^{\pi}} 
				\sum\limits_{s=1}^{n}\overline{[\phi_1(\mathbf{y})](s)}\exp(-\unitp sx)\Phi_1[f](x)\sum\limits_{\ell=1}^{n}[\phi_1(\mathbf{y})](\ell)\exp(\unitp \ell x)dx\\
				&=\frac{1}{2\pi}
				{\displaystyle \int_{-\pi}^{\pi}} \overline{\vectocpoly{\phi_1(\mathbf{y})}{x}}\Phi_1[f](x)\vectocpoly{\phi_1(\mathbf{y})}{x}dx\\
				&=\frac{1}{2\pi}
				{\displaystyle \int_{-\pi}^{\pi}} 
				\overline{\phi_1\left(\vectocpoly{\mathbf{y}}{x}\right)}\Phi_1[f](x)\phi_1\left(\vectocpoly{\mathbf{y}}{x}\right)dx
			\end{align*}
			Similarly, by Definition \ref{quatgfuncdef} and Lemma \ref{gfuncstructurelm},  
			one can show that
			\begin{align*}
				&term \ 2=\frac{1}{2\pi}\int_{-\pi}^{\pi}\phi_2\left(\vectocpoly{\mathbf{y}}{x}\right)\overline{\Phi_2[f](x)}\phi_1\left(\vectocpoly{\mathbf{y}}{x}\right)dx,\\
				&term \ 3=\frac{1}{2\pi}\int_{-\pi}^{\pi}\overline{\phi_1\left(\vectocpoly{\mathbf{y}}{x}\right)}\Phi_2[f](x)\overline{\phi_2\left(\vectocpoly{\mathbf{y}}{x}\right)}dx,\\
				&term \ 4=\frac{1}{2\pi}\int_{-\pi}^{\pi}\phi_2\left(\vectocpoly{\mathbf{y}}{x}\right)\Phi_1[f](-x)\overline{\phi_2\left(\vectocpoly{\mathbf{y}}{x}\right)}dx.
			\end{align*}
			By combing the results, we have 
			\begin{eqnarray*}
				\mathbf{y}^{*}T_n\mathbf{y} & =&term \ 1+ term \ 2+ term \ 3+ term \ 4 \\
				&=&\frac{1}{2\pi}\int_{-\pi}^{\pi}\overline{\phi_1\left(\vectocpoly{\mathbf{y}}{x}\right)}\Phi_1[f](x)\phi_1\left(\vectocpoly{\mathbf{y}}{x}\right)dx+\frac{1}{2\pi}\int_{-\pi}^{\pi}\phi_2\left(\vectocpoly{\mathbf{y}}{x}\right)\overline{\Phi_2[f](x)}\phi_1\left(\vectocpoly{\mathbf{y}}{x}\right)dx+\\
				&			&\frac{1}{2\pi}\int_{-\pi}^{\pi}\overline{\phi_1\left(\vectocpoly{\mathbf{y}}{x}\right)}\Phi_2[f](x)\overline{\phi_2\left(\vectocpoly{\mathbf{y}}{x}\right)}dx+\frac{1}{2\pi}\int_{-\pi}^{\pi}\phi_2\left(\vectocpoly{\mathbf{y}}{x}\right)\Phi_1[f](-x)\overline{\phi_2\left(\vectocpoly{\mathbf{y}}{x}\right)}dx\\
				&			=&\frac{1}{2\pi}\int_{-\pi}^{\pi}\left[\overline{\phi_1\left(\vectocpoly{\mathbf{y}}{x}\right)},\phi_2\left(\vectocpoly{\mathbf{y}}{x}\right)\right]\left[\begin{array}[c]{cc}
					\Phi_1[f](x)&\Phi_2[f](x)\\
					&\\
					\overline{\Phi_2[f](x)}&\Phi_1[f](-x)
				\end{array}\right]\left[\begin{array}[c]{c}
					\phi_1\left(\vectocpoly{\mathbf{y}}{x}\right)\\
					\\
					\overline{\phi_2\left(\vectocpoly{\mathbf{ y}}{x}\right)}
				\end{array}\right]dx\\
				&	=&\frac{1}{2\pi}\int_{-\pi}^{\pi}[\veccomplexify{\vectocpoly{\mathbf{y}}{x}}]^{*}G[f](x)\veccomplexify{\vectocpoly{\mathbf{y}}{x}}dx.
			\end{eqnarray*}
			
			For \bfrefnumber{ii}: 
			{\color{black}
				given $\mathbf{z}= 
				\left ( 
				\begin{array}{c}
					{ z}_1 \\
					{ z}_2 \\
				\end{array}
				\right )$ 
				with ${ z}_1,{ z}_2\in\mathbb{Q}_{(0,1)}^{n\times 1}$, the inverse operation of $\veccomplexify{\cdot}$ is
				\begin{equation*}
					\invveccomplexify{\mathbf{ z}}={ z}_1+\overline{{ z}_2}\unitq.
			\end{equation*}}		
			It follows from the well-known Courant–Fischer Theorem that
			\begin{align*}
				& \lambda_{\min}(\matcomplexify{T_n})\\
				&=\min\limits_{\mathbf{ z}\in\mathbb{Q}_{(0,1)}^{2n\times 1},|| \mathbf{z}||_2=1}\mathbf{ z}^{*}\matcomplexify{T_n}\mathbf{ z}\\
				&=\min\limits_{\mathbf{z}\in\mathbb{Q}_{(0,1)}^{2n\times 1},||\mathbf{z}||_2=1}[\invveccomplexify{\mathbf{z}}]^{*}T_n\invveccomplexify{\mathbf{z}}\\
				&=\min\limits_{\mathbf{z}\in\mathbb{Q}_{(0,1)}^{2n\times 1},||\mathbf{z}||_2=1}\frac{1}{2\pi}\int_{-\pi}^{\pi}[\veccomplexify{\vectocpoly{\invveccomplexify{\mathbf{z}}}{x}}]^{*}G[f](x)\veccomplexify{\vectocpoly{\invveccomplexify{\mathbf{z}}}{x}}dx.
			\end{align*}
			Since $G[f](x)$ is a 2-by-2 Hermitian matrix for each $x\in[-\pi,\pi]$,
			it is straightforward to calculate 
			\begin{equation*}
				\lambda_{\min}(G[f](x))=\check{f}(x),\quad {\rm a.e.~}x\in[-\pi,\pi].
			\end{equation*}
			By using Courant–Fischer Theorem, we see that for $x$ a.e. in $[-\pi,\pi]$, it holds
			\begin{align*}
				[\veccomplexify{\vectocpoly{\invveccomplexify{\mathbf{z}}}{x}}]^{*}G[f](x)\veccomplexify{\vectocpoly{\invveccomplexify{\mathbf{z}}}{x}}&\geq \check{f}(x)\textcolor{black}{\left|\left|\veccomplexify{\vectocpoly{\invveccomplexify{\mathbf{z}}}{x}}\right|\right|_2^2}.
			\end{align*}
			Therefore, we obtain
			\begin{align*}
				\lambda_{\min}(\matcomplexify{T_n})&\geq\min\limits_{\mathbf{z}\in\mathbb{Q}_{(0,1)}^{2n\times 1},||\mathbf{z}||_2=1}\frac{1}{2\pi}\int_{-\pi}^{\pi}\check{f}(x)\textcolor{black}{\left|\left|\veccomplexify{\vectocpoly{\invveccomplexify{\mathbf{z}}}{x}}\right|\right|_2^2}dx\\
				&\geq \essinf\limits_{x\in[-\pi,\pi]}\check{f}(x)\left( \min\limits_{\mathbf{z}\in\mathbb{Q}_{(0,1)}^{2n\times 1},||\mathbf{z}||_2=1}\frac{1}{2\pi}\int_{-\pi}^{\pi}\textcolor{black}{\left|\left|\veccomplexify{\vectocpoly{\invveccomplexify{\mathbf{z}}}{x}}\right|\right|_2^2}dx\right).
			\end{align*}
			On the other hand, we note that for $\mathbf{z}\in\mathbb{Q}_{(0,1)}^{2n\times 1}$, it holds that
			\begin{align*}
				&\frac{1}{2\pi}\int_{-\pi}^{\pi}\textcolor{black}{\left|\left|\veccomplexify{\vectocpoly{\invveccomplexify{\mathbf{z}}}{x}}\right|\right|_2^2}dx \\
				=&\frac{1}{2\pi}\int_{-\pi}^{\pi}\left|\vectocpoly{\invveccomplexify{\mathbf{z}}}{x}\right|^2dx\\
				=&\frac{1}{2\pi}\int_{-\pi}^{\pi}\sum\limits_{s=1}^{n}\sum\limits_{\ell=1}^{n}\overline{[\invveccomplexify{\mathbf{z}}](s)}[\invveccomplexify{\mathbf{z}}](\ell)\exp(\unitp (\ell-s)x) dx\\
				=&\sum\limits_{s=1}^{n}\left|[\invveccomplexify{\mathbf{z}}](s)\right|^2=||\invveccomplexify{\mathbf{z}}||_2^2=||{ z}_1||_2^2+||{ z}_2||_2^2=||\mathbf{z}||_2^2.
			\end{align*}
			Therefore, we have 
			\begin{equation*}
				\lambda_{\min}(\matcomplexify{T_n})\geq  \essinf\limits_{x\in[-\pi,\pi]}\check{f}(x)=\essinf\limits_{x\in[-\pi,\pi]}\lambda_{\min}(G[f](x)).
			\end{equation*}
			Similarly, one can show that
			\begin{equation*}
				\lambda_{\max}(\matcomplexify{T_n})\leq  \esssup\limits_{x\in[-\pi,\pi]}\hat{f}(x)=\esssup\limits_{x\in[-\pi,\pi]}\lambda_{\max}(G[f](x)).
			\end{equation*}
			
				For \bfrefnumber{iii}, let $\mu_0=\lambda_{\min}(T_n)$. Then, there exists a nonzero vector $\mathbf{z}_0\in\mathbb{Q}^{n\times 1}$ such that
				$T_n\mathbf{z}_0=\mu_0\mathbf{z}_0$. Thus, $\mathbf{z}_0^{*}T_n\mathbf{z}_0=\mu_0\mathbf{z}_0^{*}\mathbf{z}_0$. 
				We find that 
				\begin{eqnarray*}
					& &\mathbf{z}_0^{*}T_n\mathbf{z}_0-\mu_0\mathbf{z}_0^{*}\mathbf{z}_0 \\
					&=&\frac{1}{2\pi}\int_{-\pi}^{\pi}[\veccomplexify{\vectocpoly{\mathbf{z}_0}{x}}]^{*}G[f](x)\veccomplexify{\vectocpoly{\mathbf{z}_0}{x}}dx- 
					\frac{\mu_0}{2\pi}\int_{-\pi}^{\pi}[\veccomplexify{\vectocpoly{\mathbf{z}_0}{x}}]^{*}\veccomplexify{\vectocpoly{\mathbf{z}_0}{x}}dx\\
					&\geq&\frac{1}{2\pi}\int_{-\pi}^{\pi}[\veccomplexify{\vectocpoly{\mathbf{z}_0}{x}}]^{*}\check{f}(x)\veccomplexify{\vectocpoly{\mathbf{z}_0}{x}}dx-\frac{\mu_0}{2\pi}\int_{-\pi}^{\pi}[\veccomplexify{\vectocpoly{\mathbf{z}_0}{x}}]^{*}\veccomplexify{\vectocpoly{\mathbf{z}_0}{x}}dx\\
					&=&\frac{1}{2\pi}\int_{-\pi}^{\pi}[\veccomplexify{\vectocpoly{\mathbf{z}_0}{x}}]^{*}[\check{f}(x)-\mu_0]\veccomplexify{\vectocpoly{\mathbf{z}_0}{x}}dx\\
					&=&\frac{1}{2\pi}\int_{-\pi}^{\pi}\textcolor{black}{||\veccomplexify{\vectocpoly{\mathbf{z}_0}{x}}||_2^2}[\check{f}(x)-\mu_0]dx.
				\end{eqnarray*}
				By \bfrefnumber{ii}, we have already known that $\mu_0\geq \essinf\limits_{x\in[-\pi,\pi]}\check{f}(x)$. We show $\mu_0>\essinf\limits_{x\in[-\pi,\pi]}\check{f}(x)$ by contradiction. Suppose $\mu_0=\essinf\limits_{x\in[-\pi,\pi]}\check{f}(x)$. Since  there is no constant function equal to $\check{f}$  a.e. on $[-\pi,\pi]$, the measure of the set $\Omega_0:=\{x\in[-\pi,\pi]|\check{f}(x)>\mu_0\}$ is larger than zero. \textcolor{black}{Moreover, since $||\veccomplexify{\vectocpoly{\mathbf{ z}_0}{x}}||_2^2$ is a nonzero nonnegative trigonometric polynomial,  the measure of the set $\Omega_1:=\{x\in\Omega_0|||\veccomplexify{\vectocpoly{\mathbf{ z}_0}{x}}||_2^2>0\}$ is larger than zero.} Then,
				\begin{align*}
					\frac{1}{2\pi}\int_{-\pi}^{\pi}|\veccomplexify{\vectocpoly{\mathbf{z}_0}{x}}|^2[\check{f}(x)-\mu_0]dx&=\frac{1}{2\pi}\int_{\Omega_0}|\veccomplexify{\vectocpoly{\mathbf{z}_0}{x}}|^2[\check{f}(x)-\mu_0]dx\\
					&\geq \frac{1}{2\pi}\int_{\Omega_1}|\veccomplexify{\vectocpoly{\mathbf{z}_0}{x}}|^2[\check{f}(x)-\mu_0]dx>0,
				\end{align*}
				which contradicts with the fact that $0\geq \frac{1}{2\pi}\int_{-\pi}^{\pi}|\veccomplexify{\vectocpoly{\mathbf{z}_0}{x}}|^2[\check{f}(x)-\mu_0]dx$. Hence, the assumption $\mu_0=\essinf\limits_{x\in[-\pi,\pi]}\check{f}(x)$ is not valid. That means $\mu_0>\essinf\limits_{x\in[-\pi,\pi]}\check{f}(x)$ is true. 
				Similarly, if there is no constant function equal to $\hat{f}$  a.e. on $[-\pi,\pi]$, one can show that $\esssup\limits_{x\in[-\pi,\pi]}\hat{f}(x)>\lambda_{\max}(T_n)$.
				The proof is complete.
			\end{proof}
			
			\begin{remark}\label{blocktoepremk}
				In the above theorem, we relate the value of ${\bf y}^{*}T_n{\bf y}$ to an integration of the matrix-valued symbol $G[f]$ (See Theorem \ref{qhermitmatrayleightogfuncthm}\bfrefnumber{i}). This result is quite different from the counterpart of complex Hermitian Toeplitz matrix because $f$ is quaternion-valued. 
				Indeed, it is well-known that if $\{T_n\}_{n\in\mathbb{N}^{+}}$ is a sequence of  complex Hermitian Toeplitz matrices, then the integration in Theorem \ref{qhermitmatrayleightogfuncthm}\bfrefnumber{i} would be replaced by
				\begin{equation*}
					\frac{1}{2\pi}\int_{-\pi}^{\pi}|\mathcal{F}({\bf y},x)|^2f(x)dx,
				\end{equation*}
				with real valued $f$. 
				We would like to state that 
				Theorem \ref{qhermitmatrayleightogfuncthm}\bfrefnumber{ii}-\bfrefnumber{iii} can be derived from some classical theories for block Toeplitz matrix in the literature (see,e.g., \cite{serra1999spectral}), since $G[f]$ can be seen as a symbol for a sequence of block Toeplitz matrices with $2\times 2$ blocks.
			\end{remark}
			
			Because we are interested in the spectra of Hermitian quaternion Toeplitz matrices $T_n$, we give the conditions on $f$ 
			such that $T_n$ are positive definite. 
			
			\begin{corollary}\label{tnhpdcorol}
				Let $f$ be a generating function defined in Definition \ref{quatgfuncdef} for the quaternion Toeplitz Hermitian matrices $\{ T_n \}$. 
				Suppose the following conditions hold:
				\begin{description}
					\item[(i)]$|\Phi_2[f](x)|^2\leq \Phi_1[f](x)\Phi_1[f](-x)$ and $\Phi_1[f](x)\geq 0$ hold a.e. on $[-\pi,\pi]$;
					\item[(ii)] the set $\{|\Phi_2[f](x)|^2< \Phi_1[f](x)\Phi_1[f](-x) \ | \ x\in[-\pi,\pi]\}$ has a positive measure.
				\end{description}
				Then $T_n$ is HPD for each $n$.
			\end{corollary}
			
			\begin{proof}
				Recall the definition of $\check{f}$ given in Theorem \ref{qhermitmatrayleightogfuncthm}. Then,{\color{black}
					\begin{align*}
						& \check{f}(x)\geq 0\\
						&\Longleftrightarrow \Phi_1[f](x)+\Phi_1[f](-x)-\sqrt{|\Phi_1[f](x)-\Phi_1[f](-x)|^2+4|\Phi_2[f](x)|^2}\geq 0\\
						&\Longleftrightarrow [\Phi_1[f](x)+\Phi_1[f](-x)]^2\geq (\Phi_1[f](x)-\Phi_1[f](-x))^2+4|\Phi_2[f](x)|^2\\
						&~\qquad{\rm~and~}\Phi_1[f](x)+\Phi_1[f](-x)\geq 0,\\
						&\Longleftrightarrow [\Phi_1[f](x)\Phi_1[f](-x)]\geq |\Phi_2[f](x)|^2{\rm~and~}\Phi_1[f](x)+\Phi_1[f](-x)\geq 0,
					\end{align*}
					where the second equivalence comes from 
						the fact that $\Phi_1[f]$ is real valued. Note that $|\Phi_2[f](x)|^2\leq \Phi_1[f](x)\Phi_1[f](-x)$ indicates that $\Phi_1[f](x)$ and $\Phi_1[f](-x)$ have the same sign, under which  $\Phi_1[f](x)+\Phi_1[f](-x)\geq 0$ is equivalent to $\Phi_1[f](x)\geq 0$.  }From the discussion above, we conclude that 
					\begin{align}
						&\check{f}(x)> 0\Longleftrightarrow |\Phi_2[f](x)|^2< \Phi_1[f](x)\Phi_1[f](-x){\rm~and~}\Phi_1[f](x)\geq 0;\\
						&\check{f}(x)=0\Longleftrightarrow  |\Phi_2[f](x)|^2= \Phi_1[f](x)\Phi_1[f](-x){\rm~and~}\Phi_1[f](x)\geq 0.\label{checkgnonegeq1}
					\end{align}
					Hence, the condition in \bfrefnumber{i} implies that
					\begin{equation*}
						\essinf\limits_{x\in[-\pi,\pi]}\check{f}\geq 0.
					\end{equation*}
					If $\essinf\limits_{x\in[-\pi,\pi]}\check{f}> 0$, then Theorem \ref{qhermitmatrayleightogfuncthm}\bfrefnumber{ii} implies that $T_n$ is HPD for each $n$. 
					
					It remains to the discuss the case $\essinf\limits_{x\in[-\pi,\pi]}\check{f}=0$. We shall show by contradiction that there is no constant function equal to $\check{f}$ a.e. on $[-\pi,\pi]$. Assume that there exists a constant $\nu_0$ such that $\check{f}=\nu_0$ a.e. on $[-\pi,\pi]$. Then,  $\essinf\limits_{x\in[-\pi,\pi]}\check{f}=0$ shows that $\nu_0=0$. In other words, $\check{f}=0$ a.e. on $[-\pi,\pi]$, which together with \eqref{checkgnonegeq1} implies that $|\Phi_2[f](x)|^2= \Phi_1[f](x)\Phi_1[f](-x)$ a.e. on $[-\pi,\pi]$. That means the set\\ $\{|\Phi_2[f](x)|^2< \Phi_1[f](x)\Phi_1[f](-x)|x\in[-\pi,\pi]\}$ has a zero measure which contradicts with the condition in \bfrefnumber{ii}. Therefore, the assumption is not valid.  Then Theorem \ref{qhermitmatrayleightogfuncthm}\bfrefnumber{iii} implies that
					\begin{equation*}
						0=\essinf\limits_{x\in[-\pi,\pi]}\check{f}<\lambda_{\min}(T_n),
					\end{equation*}
					holds for each $n$. In other words, $T_n$ is HPD for each $n$.
					The proof is complete.
				\end{proof}
				
				\begin{remark}
					From Theorem \ref{qhermitmatrayleightogfuncthm}{\bf (ii)}, we also see that in the case of $T_n\in\mathbb{Q}_{(0,1)}^{n\times n}$ (i.e., $\Phi_2[f]\equiv 0$), it holds that
					\begin{align*}
						& \check{f}(x) \\
						=& \frac{1}{2}\left[\Phi_1[f](x)+\Phi_1[f](-x)-|\Phi_1[f](x)-\Phi_1[f](-x)|\right]=\min\{\Phi_1[f](x),\Phi_1[f](-x)\}.
					\end{align*}
					That means, in the case of $T_n\in\mathbb{Q}_{(0,1)}^{n\times n}$, $\Phi_1[f](x)>0$ for $x\in[-\pi,\pi]$ is sufficient to guarantee that $\check{f}(x)>0$ for $x\in[-\pi,\pi]$ and thus to guarantee that $T_n$ is HPD for all $n$. This coincides with the intuition built from the classical theory of generating function for Hermitian complex Toeplitz matrices. 
				\end{remark}
				
				Another important theorem on spectral distribution of Hermitian Toeplitz matrices is the Grenander-Szeg\"{o} Theorem.
				Here we give a quaternion-version Grenander-Szeg\"{o} theorem based on 
				quaternion-valued generating functions. 
					The proof is analogous to that of the Grendander-Szego theorem in the complex field. We present the details in Appendix A.
					
					\begin{theorem}\label{quatszegothm}
						Let $f$ be a generating function defined in Definition \ref{quatgfuncdef} for the quaternion Toeplitz Hermitian matrices $\{ T_n \}_{n\in\mathbb{N}^{+}}$. 
						\begin{description}
							\item[(i)] Suppose in addition $f\in L^{\infty}([-\pi,\pi])$. Then, for any continuous function $F$ on the closed interval $[\check{a},\hat{a}]$, it holds that
							\begin{equation*}
								\lim\limits_{n\rightarrow\infty}\frac{1}{n}\sum\limits_{s=1}^{n}F(\lambda_{s}(T_n))=\frac{1}{4\pi}\int_{-\pi}^{\pi}\left[F\left(\hat{f}(x)\right)+F\left(\check{f}(x)\right)\right]dx,
							\end{equation*}
							where $\check{f}$ and $\hat{f}$ are defined in Theorem \ref{qhermitmatrayleightogfuncthm}{\bf (ii)}; here,
							$\check{a}=\essinf\limits_{x\in[-\pi,\pi]} \check{f}(x)$ and $\hat{a}=\esssup\limits_{x\in[-\pi,\pi]} \hat{f}(x)$ are both finite numbers since $\check{f},\hat{f}\in L^{\infty}([-\pi,\pi])$ {\color{black}are} guaranteed by $f\in L^{\infty}([-\pi,\pi])$.
							\item[(ii)] Suppose in addition $f\in L^2([-\pi,\pi])$. Then, for any continuous function $F$ with compact support in $\mathbb{R}$, it holds that
							\begin{equation*}
								\lim\limits_{n\rightarrow\infty}\frac{1}{n}\sum\limits_{s=1}^{n}F(\lambda_{s}(T_n))=\frac{1}{4\pi}\int_{-\pi}^{\pi}\left[F(\hat{f}(x))+F(\check{f}(x))\right]dx,
							\end{equation*}
							where $\check{f}$ and $\hat{f}$ are defined in Theorem \ref{qhermitmatrayleightogfuncthm}{\bf (ii)}.
						\end{description}
					\end{theorem}

					\subsection{Quaternion Signal Processing}\label{qsignalsubsec}
					
					In this subsection, we give examples of generating functions arising from 
					quaternion signal processing. 
					
					We consider a discrete-time quaternion signal 
					$\{ x(t) = x_0(t) + x_1(t)\unitp+x_2(t)\unitq+x_3(t)\unitr \}$ with 
					a finite 2nd-moment
					$\mathcal{E}(|x(t)|^2)<+\infty$ for each $t$,
					where $\mathcal{E}(\cdot)$ denotes the expectation operator,
					see \cite{took2011augmented,chen2023phase}.
					The linear prediction is to find a set of linear combination coefficients $\{\alpha_s\}_{s=1}^{n}\subset\mathbb{Q}$ such that the linear combination $\sum\limits_{s=1}^{n}\alpha_{s}x(t-s)$ best fitting $x(t)$ under the expected value:
						\begin{align}
							&\mathcal{E}\left(\left(x(t)-\sum\limits_{s=1}^{n}\alpha_{s}x(t-s)\right)\overline{\left(x(t)-\sum\limits_{s=1}^{n}\alpha_{s}x(t-s)\right)}\right)\notag\\
							&=\min\limits_{\{w_s\}_{s=1}^{n}\subset\mathbb{Q}}\mathcal{E}\left(\left(x(t)-\sum\limits_{s=1}^{n}w_{s}x(t-s)\right)\overline{\left(x(t)-\sum\limits_{s=1}^{n}w_{s}x(t-s)\right)}\right).\label{alphasoptimiz}
						\end{align}
					Once  $\{\alpha_s\}_{s=1}^{n}$ is computed, one can then predict the value of the signal at next time step by computing the linear combination of the most recent $n$ many observations  with $\{\alpha_s\}_{s=1}^{n}$ as the left-hand side coefficients.
					It is straightforward to verify that \eqref{alphasoptimiz} is equivalent to the systems of equations
					\begin{equation*}
						\mathcal{E}\left(\left(x(t)-\sum\limits_{s=1}^{n}\alpha_{s}x(t-s)\right)\overline{x(t-s')}\right)=0, \ s'=1,2,...,n,
					\end{equation*}
					{\color{black} which means
						\begin{align} \label{linearsys}
							\sum\limits_{s=1}^{n}\mathcal{E}\left(x(t-s')\overline{x(t-s)}\right)\overline{\alpha_s}&=\overline{\sum\limits_{s=1}^{n}\alpha_s\mathcal{E}\left(x(t-s)\overline{x(t-s')}\right)}\notag\\
							&=\overline{\mathcal{E}\left(x(t)\overline{x(t-s')}\right)}=\mathcal{E}\left(x(t-s')\overline{x(t)}\right), \ s'=1,2,...,n.
					\end{align}}
					Suppose
					$\mathcal{E}\left(x(t)\overline{x(t\textcolor{black}{+}s)}\right)$ is a function only of the time-lag $s$, i.e., there exists a covariance function of $s$, $\eta(s)$ such that
					\begin{equation}\label{covarianceshiftinvarianteq}
						\mathcal{E}\left(x(t)\overline{x(t+s)}\right)=\eta(s),\quad s=0,1, 2,\cdots.
					\end{equation}
					With \eqref{covarianceshiftinvarianteq}, $\mathcal{E}\left(x(t)\overline{x(t+s)}\right)$ for negative $s$ is also a function only of time-lag $s$. This is due to the following facts
					\begin{align*}
						\mathcal{E}\left(x(t)\overline{x(t+s)}\right)&=\overline{\mathcal{E}\left(x(t+s)\overline{x(t+s+(-s))}\right)}=\overline{\eta(-s)},\quad s=-1,-2,\cdots.
					\end{align*}
					Then (\ref{linearsys}) indicates a Hermitian quaternion Toeplitz linear system as follows:
					\begin{equation}\label{hermqtsystem}
						T_n \bm{\alpha} = \mathbf{w} = \left[\eta(1),\eta(2),...,\eta(n)\right]^{\rm T},
					\end{equation}
					where 
					$$
					[ T_n ]_{s',s} =\begin{cases}
						\mathcal{E}\left(x(t-s')\overline{x(t-s'+(s'-s))}\right)=\eta(|s'-s|),\quad s'\geq s,\\
						\mathcal{E}\left(x(t-s')\overline{x(t-s'+(s'-s))}\right)=\overline{\eta(|s'-s|)},\quad s'<s,
					\end{cases}
					$$
					and the components $\overline{\alpha_{s'}}$ $(s'=1,2,...,n)$  
					of the unknown vector $\alpha$.
					When $\{\eta(s)\}_{s=0}^{\infty}$ is absolutely summable, it is straightforward to verify that 
					\begin{equation} \label{genf}
						f(\theta)=\eta(0)+\sum\limits_{s=1}^{+\infty} \eta(s) \exp(\unitp s \theta)+ \overline{\eta(s)} 
						\exp(-\unitp s \theta), \ \theta \in [-\pi,\pi],
					\end{equation}
					meets Definition \ref{quatgfuncdef} as a generating function for $T_n$
					in \eqref{hermqtsystem}.
					
					{\rm
						\begin{example}\label{ar1}
							We consider a quaternionic signal  $x(t)$ generated by
								\begin{equation*}
									x(t)=\beta x(t-1)+e(t),
								\end{equation*}
								where $\beta$ is a given quaternion parameter with $|\beta|\in(0,1)$; $e(t)=e_0(t)+e_1(t)\unitp+e_2(t)\unitq+e_3(t)\unitr$ represents the Gaussian noise; $e_{v}(t)$ (for different $v$ and $t$) obeys i.i.d. Gaussian distribution with $\delta^2=1$  as variance and 0 as mean. It is straightforward to verify that
								\begin{equation*}
									x(t+s)=\beta^sx(t)+\sum\limits_{l=1}^{s}\beta^{l-1}e(t+s+1-l),\quad s=1,2,\cdots.
								\end{equation*}
								With the equalities above, one can show that
								\begin{align*}
									& \eta(s)=\mathcal{E}\left(x(t)\overline{x(t+s)}\right)=\frac{4\delta^2\bar{\beta}^s}{1-|\beta|^2},\quad s=0,1,2,\cdots.
								\end{align*}
								Then, we obtain
								$$\sum\limits_{s=0}^{+\infty}|\eta(s)|=\frac{4 \delta^2}{1-|\beta|^2}\sum\limits_{s=0}^{+\infty}|\beta|^s=\frac{4\delta^2 }{1-|\beta|^2}\times \frac{1}{1-|\beta|}=\frac{4\delta^2 }{(1-|\beta|^2)(1-|\beta|)}<+\infty,$$
								which implies that the sequence $\{\eta(s)\}_{s=0}^{+\infty}$ is absolutely summable. 
								Rewrite $\beta$ as $\beta=\beta_0+\beta_1\unitp+\beta_2\unitq+\beta_3\unitr$ with $\beta_0,\beta_1,\beta_2,\beta_3\in\mathbb{R}$. Then,
								\begin{equation}\label{mupolardecompos}
									\beta=|\beta|\exp({\bf m}\theta_0),
								\end{equation}
								{\color{black} where  $\exp({\bf m}x):=\cos(x)+{\bf m}\sin(x)$, ${\bf m}:=\frac{\beta_I}{|\beta_I|},~ \beta_I:=\beta_1\unitp+\beta_2\unitq+\beta_3\unitr$. Note that $\theta_0\in[0,\pi]$ is the unique number satisfying 
									\begin{align*}
										&\cos(\theta_0)=\beta_0/|\beta|,\quad \sin(\theta_0)=|\beta_I|/|\beta|,
									\end{align*}
									and ${\bf m}$ is a pure imaginary  unit. Hence,
									\eqref{mupolardecompos} is the so-called polar decomposition of $\beta$. Now $\beta^s$ can be expressed as
									\begin{equation*}
										\beta^s=|\beta|^s\exp({\bf m}s\theta_0),\quad s=0,1,2,\cdots.
								\end{equation*}}
								Therefore, we have 
								\begin{align*}
									&\beta^s=\phi_1(\beta^s)+\phi_2(\beta^s)\unitq,\quad \phi_1(\beta^s)=|\beta|^s\left[\cos(s\theta_0)+|\beta_I|^{-1}\beta_1\sin(s\theta_0)\unitp\right],\\
									&\phi_2(\beta^s)=\frac{|\beta|^s}{|\beta_I|}\left[\beta_2\sin(s\theta_0)+\beta_3\sin(s\theta_0)\unitp\right].
								\end{align*}
								By \eqref{genf}, we know that the Hermitian quaternion matrix $T_n$ defined in \eqref{hermqtsystem} corresponding to Example \ref{ar2} is generated by a function $f=\Phi_1[f]+\Phi_2[f]\unitq$ with
								\begin{align*}
									\Phi_1[f](\theta)=&\eta(0)+\sum\limits_{s=1}^{+\infty}\phi_1(\eta(s))\exp(\unitp s\theta)+\phi_1\left(\overline{\eta(s)}\right)\exp(-\unitp s\theta)\\
									=&\frac{4 \delta^2}{1-|\beta|^2}\bigg[-1+\frac{(1+\beta_1/|\beta_I|)(1-|\beta|\cos(\theta_0-\theta))}{[1-|\beta|\cos(\theta_0-\theta)]^2+|\beta|^2\sin^2(\theta_0-\theta)}\\
									&+\frac{(1-\beta_1/|\beta_I|)(1-|\beta|\cos(\theta+\theta_0))}{[1-|\beta|\cos(\theta+\theta_0)]^2+|\beta|^2\sin^2(\theta+\theta_0)}
									\bigg],\\
									\Phi_2[f](\theta)=&\sum\limits_{s=1}^{+\infty}\phi_2\left(\overline{\eta(s)}\right)\exp(\unitp s\theta)+\phi_2(\eta(s))\exp(-\unitp s\theta)\\
									=&\frac{4 \delta^2(\beta_3-\beta_2\unitp)}{(1-|\beta|^2)|\beta_I|}\bigg[\frac{1-|\beta|\cos(\theta+\theta_0)}{[1-|\beta|\cos(\theta+\theta_0)]^2+|\beta|^2\sin^2(\theta+\theta_0)}\\
									&-\frac{1-|\beta|\cos(\theta_0-\theta)}{[1-|\beta|\cos(\theta_0-\theta)]^2+|\beta|^2\sin^2(\theta_0-\theta)}\bigg].
								\end{align*}
									With $f$ given above, we can easy get the corresponding $\check{f}$ from its definition in Theorem \ref{qhermitmatrayleightogfuncthm}):
									$$
									\check{f}(\theta)=\frac{4\delta^2}{1+|\beta|^2-2|\beta|\cos(|\theta|+\theta_0)},\quad 
									\theta \in [-\pi, \pi].
									$$	
									It is straightforward to verify $\check{f}$ to be a positive function, which guarantees the covariance matrix $T_n$ to be an HPD quaternion matrix.
								\end{example}
								
								\begin{example}\label{ar2}
										In this example, we consider a quaternionic noise $x(t)$ generated by 
										\begin{equation*}
											x(t)=\beta e(t-1)+e(t),
										\end{equation*}
										where $\beta$ is a given quaternion parameter with $|\beta|\in(0,1)$; $e(t)=e_0(t)+e_1(t)\unitp+e_2(t)\unitq+e_3(t)\unitr$ represents the Gaussian noise; $e_{v}(t)$ (for different $v$ and $t$) obeys i.i.d. Gaussian distribution with $\delta^2=1$  as variance and 0 as mean. It is straightforward to verify that
										\begin{equation*}
											\eta(s)=\mathcal{E}\left(x(t)\overline{x(t+s)}\right)=\begin{cases}
												4\delta^2(|\beta|^2+1),\quad s=0,\\
												4\delta^2\bar{\beta},\quad s=1,\\
												0,\quad s\geq 2.
											\end{cases}.
										\end{equation*}
										Clearly, the sequence $\{\eta(s)\}_{s=0}^{+\infty}$ is absolutely summable. 
										The generating function can be constructed, which is given by
										$$
										f(\theta)= \Phi_1[f](\theta) + \Phi_2[f](\theta) \unitq, \ \theta\in[-\pi,\pi],
										$$	
										where 	
										$$ 
										\Phi_1[f](\theta)=4\delta^2(|\beta|^2+1)+8\delta^2\left[\beta_0\cos(\theta)+\beta_1\sin(\theta)\right], \ \theta\in[-\pi,\pi],
										$$
										$$ 
										\Phi_2[f](\theta)=[8\delta^2(-\beta_3+\beta_2\unitp)\sin(\theta)], \ \theta\in[-\pi,\pi],
										$$
										and $\beta=\beta_0+\beta_1\unitp+\beta_2\unitq+\beta_3\unitr$.
										With $f$ defined above, by the definition of its corresponding function $\check{f}$  in Theorem \ref{qhermitmatrayleightogfuncthm}\bfrefnumber{ii}), we have,
										$$
										\check{f}(\theta)=4\delta^2\big[1+|\beta|^2+2|\beta|\cos(\theta_0+|\theta|)\big], \quad \theta \in [-\pi,\pi].
										$$
										It is straightforward to verify that  $\check{f}$ is a positive function, which guarantees the covariance matrix $T_n$ to be positive definite.
									\end{example}
									
									\section{Circulant Preconditioned Conjugate Gradient Method}
									Theory of circulant preconditioners for preconditioning real or complex Toeplitz linear systems has been well studied in the literature; see, e.g., \cite{ng2004iterative,chan2007introduction,jin2008survey,chan1996conjugate,estatico2012note,huckle1993some,strela1996circulant,chan1989circulant,strang1986proposal,pestana2015preconditioned}. 
									However, circulant preconditioners for quaternion Toeplitz systems are rarely to see. The main purpose of this section is to demonstrate how the classical theory of circulant preconditioner for complex HPD Toeplitz matrix is extended to that for quaternion HPD Toeplitz matrix. 
									To know more about the classical preconditioning theories for block Toeplitz matrices, one may refer to \cite{chan1994circulant,serra1998asymptotic,huckle2007preconditioning,huckle2012compact,an2024aggregation,jin1995note} and the references therein.
									
									\subsection{Circulant Preconditioners}	
									A Toeplitz matrix $C_n\in\mathbb{Q}^{n\times n}$ of the following form is called a circulant matrix.
									\begin{equation}\label{circmatform}
										C_n=\left[
										\begin{array}
											[c]{cccc}
											c_0  &  c_{n-1}&\ldots   & c_1 \\
											c_1  &  c_0 & \ddots   &\vdots \\
											\vdots&\ddots &\ddots&c_{n-1}\\
											c_{n-1}   & \ldots &  c_1  &  c_{0} 
										\end{array}
										\right]
									\end{equation}
									which its rows are composed of the same elements and each row is rotated one element to the right relative to the preceding row.
										Let $\mathbf{v}:=(c_0,c_1,\cdots,c_{n-1})^T$, the first column of $C$, we notice that $C$ can be identified by $\mathbf{v}$. For simplicity, we denote $C_n$ by $\vectocirc{\mathbf{v}}$. 
										Given a general $n\times n$ Toeplitz matrix $T_n=[t_{s-l}]_{s,l=1}^{n}$, let
										\begin{align*}
											&\strangfcoloftoep{T_n}:=(t_0,t_1,...,t_{\lfloor (n-1)/2\rfloor},t_{-\lfloor (n-1)/2\rfloor},...,t_{-2},t_{-1})^{\rm T}, ~ \text{if $n$ is odd} ;\\
											&\strangfcoloftoep{T_n}:=(t_0,t_1,...,t_{\lfloor (n-1)/2\rfloor},0,t_{-\lfloor (n-1)/2\rfloor},...,t_{-2},t_{-1})^{\rm T}, ~ \text{if $n$ is even.}
										\end{align*}
										Then, the Strang circulant preconditioner denoted by $c(T_n)$ for a Toeplitz matrix 
										is defined by
										\begin{equation*}
											c(T_n):=\vectocirc{\strangfcoloftoep{T_n}}.
										\end{equation*}
										Apparently,  $c(T_n)$ is Hermitian whenever $T_n$ is a Hermitian Toeplitz matrix. 
										Here we employ the Strang circulant preconditioner as an example, other circulant preconditioners can be 
										constructed and studied similarly. 
										
										For any $m\in\mathbb{N}^{+}$, denote 
										\begin{equation*}
										\textcolor{black}{F_{\unitp,m}:=\frac{1}{\sqrt{m}}\left[\exp\left(\frac{2\pi\unitp (s-1)(l-1)}{m}\right)\right]_{s,l=1}^{m}}.
										\end{equation*}
										$F_{\unitp,m}$ is called a quaternion discrete Fourier matrix, which is unitary; see, e.g., \cite{pan2024block}.
										As shown in \cite{pan2024block}, for general $\mathbf{v}\in\mathbb{Q}^{m\times 1}$, $F_{\unitp,m}\vectocirc{\mathbf{v}}F_{\unitp,m}^{*}$ has 
										an $X$-shape sparse pattern shown as follows:
										\begin{equation*}
											F_{\unitp,m}\vectocirc{\mathbf{v}}F_{\unitp,m}^{*}=\left[\begin{array}[c]{cccccc}
												*&0&\ldots&\ldots&\ldots&0\\
												0&*&0&\ldots&0&*\\
												\vdots&0&\ddots&& \iddots&0\\
												\vdots&\vdots&& &&\vdots\\
												\vdots&0&\iddots&&\ddots&0\\
												0&*&0&\ldots&0&*
											\end{array}\right].
										\end{equation*}
										from which it is clear to see that for general $\mathbf{v}\in\mathbb{Q}^{m\times 1}$, applying Fourier transform to the first column of $\vectocirc{\mathbf{v}}$ would no longer obtain a vector with components in $\sigma(\vectocirc{\mathbf{v}})$, the right spectrum of $\vectocirc{\mathbf{v}}$. 
										Interestingly, the next lemma states that $c(T_n)$ can be diagonalized by 
										$F_{\unitp,m}$ and its eigenvalues are the eigenvalues of 
										2-by-2 block matrices where their entries are the values evaluated at the partial sum of quaternion-valued generating functions. 
										
										\begin{lemma}\label{strangcireiglimitlm}
											Suppose $\{\eta(s)\}_{s=0}^{\infty}$ is absolutely summable and 
											$T_n$ is generated by 
											$f(x)=\eta(0)+\sum\limits_{s=1}^{+\infty} \eta(s) \exp(\unitp sx)+ \overline{\eta(s)} \exp(-\unitp sx)$ for $x \in [\pi,\pi]$
											$(cf. (\ref{genf}))$.
											\begin{description}
												\item[(i)] $c(T_n)$ is Hermitian and the eigenvalues of $c(T_n)$ are real. Moreover, we have 
												\begin{equation*}
													\sigma(c(T_n))= \bigcup_{s=0}^{m}\sigma\left(G
													\left[f_{m}\right]\left(\frac{2\pi s}{n}\right)\right)
													\ {\rm with} \ m = \left\lfloor\frac{n}{2}\right\rfloor,
												\end{equation*}
												where $G[\cdot]$ is defined in Theorem \ref{qhermitmatrayleightogfuncthm}, and 
												$f_{m}$ is the $m$-th partial sum of $f$, i.e., 
												$f_m(x) = \eta(0)+\sum\limits_{s=1}^{m} 
												\eta(s) \exp(\unitp s x)+ \overline{\eta(s)} \exp(-\unitp s x)$ for $x \in [\pi,\pi]$.
												\item[(ii)]
												\begin{equation*}
													\lim\limits_{n\rightarrow\infty}\lambda_{\min}(c(T_n))=\min\limits_{x\in[0,\pi]}\check{f}(x),\qquad \lim\limits_{n\rightarrow\infty}\lambda_{\max}(c(T_n))=\max\limits_{x\in[0,\pi]}\hat{f}(x),
												\end{equation*}
												where $\check{f}$ and $\hat{f}$ are defined in Theorem \ref{qhermitmatrayleightogfuncthm}.
											\end{description}		
										\end{lemma}
										
										\begin{proof}
											By construction, $c(T_n)$ is Hermitian as $T_n$ is Hermitian.
											\begin{align*}
												\matcomplexify{c(T_n)}=\matcmplxexpress{c(T_n)}=\left[\begin{array}[c]{cc}
													c(\phi_1(T_n))&c(-\phi_2(T_n))\\
													&\\
													c\left(\overline{\phi_2(T_n)}\right)&c\left(\overline{\phi_1(T_n)}\right)
												\end{array}\right].
											\end{align*}
											With the given $f$, we know that  		
											$$
											\Phi_1[f](x) = t_0 + \sum_{s=1}^{\infty} \phi_1(t_s) \exp( \unitp s x) + \phi_1( \bar{t}_s ) \exp( - \unitp s x),
											$$
											and 
											$$
											\Phi_2[f](x) = \sum_{s=1}^{\infty} \phi_2(\bar{t}_s) \exp( \unitp s x) + \phi_2( t_s ) \exp( - \unitp s x).
											$$
											Since $\phi_1(T_n)$ and $\phi_2(T_n)$ are Toeplitz matrices with entries 
											in $\mathbb{Q}_{(0,1)}$ which is isomorphic to the complex field,
											the corresponding circulant matrices 
											$c(\phi_1(T_n))$, $c(-\phi_2(T_n))$,
											$c\left(\overline{\phi_2(T_n)}\right)$ and $c\left(\overline{\phi_1(T_n)}\right)$
											can be diagonalized by $F_{\unitp,n}$. 
										\textcolor{black}{	Their eigenvalues
											are given by 
											$\Phi_1[f_m] \left(\frac{2\pi s}{n}\right)$, 
											$\Phi_2[f_m] \left(\frac{2\pi s}{n}\right)$, 
											$\overline{\Phi_2[f_m] \left(\frac{2\pi s}{n}\right)}$
											and 
											$\Phi_1[f_m] \left(-\frac{2\pi s}{n}\right)$ 
											($s=0,1,...,n-1$) respectively,
											see for instance \cite{chan1996conjugate,ng2004iterative,chan2007introduction}.
											Then, it is easy to see that $\matcomplexify{c(T_n)}$ is unitarily similar to the following $2n$-by-$2n$ matrix
											\begin{align*}
												\left[\begin{array}[c]{cc}
													{\rm diag}
													\left( \Phi_1[f_m] \left(\frac{2\pi s}{n}\right)\right)_{s=0}^{n-1}	&	
													{\rm diag}
													\left( \Phi_2[f_m] \left(\frac{2\pi s}{n}\right) \right)_{s=0}^{n-1}\\
													& \\
													{\rm diag}
													\left( \overline{\Phi_2[f_m] \left(\frac{2\pi s}{n}\right)} \right)_{s=0}^{n-1}	
													&
													{\rm diag}
													\left( \Phi_1[f_m] \left(-\frac{2\pi s}{n}\right) \right)_{s=0}^{n-1}
												\end{array}\right].
											\end{align*}
											The above matrix is of the diagonal block form, and its eigenvalues are equal to the eigenvalues of the 
											following matrices: 
											$$
											\left[\begin{array}[c]{cc}
												\Phi_1[f_m] \left(\frac{2\pi s}{n}\right)	&
												\Phi_2[f_m] \left(\frac{2\pi s}{n}\right)  \\
												& \\
												\overline{\Phi_2[f_m] \left(\frac{2\pi s}{n}\right)} 	&
												\Phi_1[f_m] \left(-\frac{2\pi s}{n}\right) 
											\end{array}\right]
											= G \left [ f_m \right]\left(\frac{2\pi s}{n}\right), \quad s=0,1,...,n-1.
											$$
												}
											Therefore, we have 
											\begin{equation}\label{cstnspeceq1}
												\sigma(c(T_n))=\bigcup_{s=0}^{n-1}\sigma\left( G\left[
												f_m \right] \left(\frac{2\pi s}{n}\right)\right).
											\end{equation}
											On the other hand, it is straightforward to derive
											\begin{align*}
												&G \left [ 
												f_m \right]
												\left(\frac{2\pi (n-s)}{n}\right)\\
												&=G\left[
												f_m\right] 
												\left(\frac{-2\pi s}{n}\right)\\
												&=\left[\begin{array}[c]{cc}
													\Phi_1[f_m]
													\left(-\frac{2\pi s}{n}\right)&
													\Phi_2[f_m]
													\left(-\frac{2\pi s}{n}\right)\\
													& \\
													\overline{\Phi_2[f_m]}
													\left(-\frac{2\pi s}{n}\right) &
													\Phi_1[f_m]
													\left(\frac{2\pi s}{n}\right)
												\end{array}\right]\\
												&=\left[\begin{array}[c]{cc}
													\Phi_1[f_m]
													\left(-\frac{2\pi s}{n}\right)&
													-\Phi_2[f_m]
													\left(\frac{2\pi s}{n}\right)\\
													& \\
													-\overline{\Phi_2[f_m]}
													\left(\frac{2\pi s}{n}\right) &
													\Phi_1[f_m]
													\left(\frac{2\pi s}{n}\right)
												\end{array}\right]\\
												&=\left[\begin{array}[c]{cc}
													&1\\
													1&
												\end{array}\right]^{-1}\underbrace{\left[\begin{array}[c]{cc}
														\Phi_1[f_m]
														\left(\frac{2\pi s}{n}\right)&
														-\overline{\Phi_2[f_m]}
														\left(\frac{2\pi s}{n}\right) \\
														& \\
														-\Phi_2[f_m]
														\left(\frac{2\pi s}{n}\right)&
														\Phi_1[f_m]
														\left(-\frac{2\pi s}{n}\right)
													\end{array}\right]}_{:=H_s}\left[\begin{array}[c]{cc}
													&1\\
													1&
												\end{array}\right], \quad s=1,...,n-1.
											\end{align*}
											By taking conjugate of $H_s$, one obtain that
											\begin{align*}
												\bar{H}_s&=\left[\begin{array}[c]{cc}
													\Phi_1[f_m]
													\left(\frac{2\pi s}{n}\right)&
													 -\Phi_2[f_m]
													 \left(\frac{2\pi s}{n}\right)\\
													& \\
													-\overline{\Phi_2[f_m]}
													\left(\frac{2\pi s}{n}\right)&
													\Phi_1[f_m]
													\left(-\frac{2\pi s}{n}\right)
												\end{array}\right]\\	
												&=
												\left[\begin{array}[c]{cc}
													-1&\\
													&1
												\end{array}\right]^{-1}
												\left[\begin{array}[c]{cc}
													\Phi_1[f_m] \left(\frac{2\pi s}{n}\right)	&
													\Phi_2[f_m] \left(\frac{2\pi s}{n}\right)  \\
													& \\
													\overline{\Phi_2[f_m] \left(\frac{2\pi s}{n}\right)} 	&
												\Phi_1[f_m] \left(-\frac{2\pi s}{n}\right)
												\end{array}\right]
												\left[\begin{array}[c]{cc}
													-1&\\
													&1
												\end{array}\right] \\
												&=\left[\begin{array}[c]{cc}
													-1&\\
													&1
												\end{array}\right]^{-1}G\left[f_{m}\right]\left(\frac{2\pi s}{n}\right)\left[\begin{array}[c]{cc}
													-1&\\
													&1
												\end{array}\right],\quad s=1,2,...,n-1.
											\end{align*}
												Note that $\sigma(\cdot)$ is right spectrum. Since $H_s$ is a Hermitian matrix, $\sigma(H_s)=\sigma(\bar{H}_s)$. Then, matrix similarities imply that
												\begin{equation*}
													\sigma\left(G\left[f_m\right]\left(\frac{2\pi (n-s)}{n}\right)\right)=\sigma(H_s)=\sigma(\bar{H}_s)=\sigma\left(G\left[
													f_{m}\right]\left(\frac{2\pi s}{n}\right)\right),\ s=1,2,...,n-1,
												\end{equation*}
												which together with \eqref{cstnspeceq1} implies that
												\begin{equation*}
													\sigma(c(T_n))= \bigcup_{s=0}^{m}\sigma\left(G\left[f_m\right]\left(\frac{2\pi s}{n}\right)\right),
												\end{equation*}
												where the set inclusion comes from the fact that $G\left[f_m\right]\left(x\right)$ is Hermitian in $\mathbb{Q}_{(0,1)}$ for each $x$.
												
												Since $\{t_s\}_{s\in\mathbb{N}^{+}}$ is absolutely summable, it is easy to check that
												\begin{equation*}
													\lim\limits_{n\rightarrow\infty}\sup\limits_{x\in[-\pi,\pi]}| f_{\lfloor\frac{n}{2}\rfloor}(x)-f(x)|=0.
												\end{equation*}
												With the convergence above and the proven result \bfrefnumber{i}, we show that
												\begin{align*}
													\lim\limits_{n\rightarrow\infty}\lambda_{\min}(c(T_n))&=\lim\limits_{n\rightarrow\infty}\min\limits_{0\leq s\leq 
														\lfloor\frac{n}{2}\rfloor}\lambda_{\min}\left(G\left[f_{\lfloor\frac{n}{2}\rfloor}\right]\left(\frac{2\pi s}{n}\right)\right)\\
													&=\lim\limits_{n\rightarrow\infty}\min\limits_{0\leq s\leq \lfloor\frac{n}{2}\rfloor}\lambda_{\min}\left(G\left[f\right]\left(\frac{2\pi s}{n}\right)\right)\\
													&=\lim\limits_{n\rightarrow\infty}\min\limits_{0\leq s\leq \lfloor\frac{n}{2}\rfloor}\check{f}\left(\frac{2\pi s}{n}\right)=\min\limits_{x\in[0,\pi]}\check{f}(x).
												\end{align*}
												Similarly, one can show that $\lim\limits_{n\rightarrow\infty}\lambda_{\max}(c(T_n))=\max\limits_{x\in[0,\pi]}\hat{f}(x)$. The proof is complete.
											\end{proof}
											
											\subsection{The Spectra of Preconditioned Matrices}
											
											In this subsection, we show the spectra of $c(T_n)^{-1} T_n$ is clustered around 1 under the assumption that the sequence $\{t_k\}_{k\in\mathbb{N}^{+}}$ is absolutely summable. 
											
											Referring to the analysis in \cite{ng2004iterative,chan2007introduction}, one can prove show the following lemma.
											\begin{lemma}\label{tnminuscstneigcontrllm}
												For any $\epsilon>0$, there exists  $n_0$ such that for all $n>n_0$, there exists $W_n,U_{(n,n_0)}\in\mathbb{Q}_{(0,1)}^{2n\times 2n}$ such that $\matcomplexify{T_n}-\matcomplexify{c(T_n)}=W_n+U_{(n,n_0)}$ with $||W_n||_2\leq \epsilon$ and ${\rm rank}(U_{(n,n_0)})\leq 4n_0$.
											\end{lemma}
											
											\begin{theorem}\label{strangcircprecedmatspectrumthm}\textnormal{(Spectra clustering of preconditioned matrix)}
												{\color{black}Suppose $\check{f}$ defined in Theorem \ref{qhermitmatrayleightogfuncthm} is a positive function. Denote $\check{f}_{\min}:=\min\limits_{x\in[-\pi,\pi]}\check{f}(x)>0$, $\hat{f}_{\max}:=\max\limits_{x\in[-\pi,\pi]}\hat{f}(x)>0$ with $\hat{f}$ defined in Theorem \ref{qhermitmatrayleightogfuncthm}. Then, for any $\epsilon\in(0,1)$, there exists $n_0>0$ such that for all $n>n_0$, it holds that  $c(T_n)^{-1}T_n$ has at most $4n_0$ many eigenvalues in $\sigma(c(T_n)^{-1}T_n)$ lying outside  $[1-\epsilon,1+\epsilon]$ } and that
												\begin{equation*}
													\sigma\left(c(T_n)^{-1}T_n\right)\subset \bigg[\frac{2\check{f}_{\min}}{3\hat{f}_{\max}},+\infty\bigg)
												\end{equation*}
											\end{theorem}
											\begin{proof}
												See Appendix \ref{proofofcircprecedmatspecclstthm}.
											\end{proof}

											
											\subsection{The Circulant Preconditioned Conjugate Gradient Method}
										
										
										{\color{black}  The PCG method is an efficient and powerful iterative algorithm to solve HPD linear systems, that improves convergence rates compared to the standard Conjugate Gradient (CG) method by incorporating a preconditioner. The PCG solver using an HPD preconditioner $P\in\mathbb{Q}^{n\times n}$ for solving a general $n\times n$ HPD quaternion linear system
											\begin{equation}\label{generalhpdqsystem}
												A\mathbf{w}=\mathbf{z},
											\end{equation}
											follows the same framework as that used in solving HPD linear systems, which is described in detail in many relevant textbooks, such as \cite{ng2004iterative}. Hence we omit the details here. }

										%
										When applying PCG algorithm to solve \eqref{hpdqtsystem}, the dominant operation cost is on computing some matrix-vector multiplications of 
											form $P^{-1}\mathbf{z}_1$ and $T_n\mathbf{z}_2$ for some given vectors 
												$\mathbf{z}_1,\mathbf{z}_2\in\mathbb{Q}^{n\times 1}$ during each iteration. It has been proven in \cite{pan2024block} that any circulant quaternion matrix is fast block diagonalizable by means of fast Fourier transforms (FFTs), with each eigen-block being of size at most $2\times 2$. As a result, for any invertible circulant matrix $C\in\mathbb{Q}^{n\times n}$, the matrix vector product $C^{-1}\mathbf{x}$ can be fast computed within 
													$\mathcal{O}(n\log n)$ operations for any given vector $\mathbf{x}\in\mathbb{Q}^{n\times 1}$. Moreover, it is well-known that the matrix-vector $T_n\mathbf{y}$ for any given $\mathbf{y}\in\mathbb{Q}^{n\times 1}$ can be fast computed within $\operationcount{n\log n}$  by embedding the $n\times n$ Toeplitz matrix $T_n$ into a $2n\times 2n$ circulant matrix (see, e.g., \cite{ng2004iterative,chan2007introduction}). Therefore, when applying PCG algorithm with an HPD circulant matrix as preconditioner to solve \eqref{hpdqtsystem}, it requires $\operationcount{n\log n}$ operations in each iteration. 
													
													The complexity of PCG depends not only on the operation cost at each iteration but also on the number of iterations required for achieving a specific stopping criterion. Hence, the convergence rate of PCG solver is another important property that has to be investigated. {\color{black} We will discuss it in the following theorems. To start the discussion, we first need some notations. For any vector $\mathbf{x}$, we denote the $i$th component of $\mathbf{x}$ by $x{\bf (i)}$.	For any $\mathbf{x},\mathbf{y}\in\mathbb{K}^{n\times 1}$ 
															($\mathbb{K}=\mathbb{Q}$ or $\mathbb{Q}_{(0,1)}$), define 
															\begin{equation*}
																\innerprod{\mathbf{x}}{\mathbf{y}}:=\mathbf{y}^{*}\mathbf{x}=\sum\limits_{s=1}^{n}\overline{y(s)}x(s).
															\end{equation*}
																For any HPD matrix $A\in\mathbb{K}^{n\times 1}$, define
																\begin{equation}\label{bnormdef}
																	||\mathbf{x}||_{A}:=\innerprod{A\mathbf{x}}{\mathbf{x}}^{\frac{1}{2}},\quad \forall \mathbf{x}\in\mathbb{K}^{n\times 1}.
																\end{equation}
																Clearly, the above defined $||\cdot||_{B}$ is a vector norm on $\mathbb{K}^{n\times 1}$.
																
																Let  $\mathbf{w}^{(k)}~(k\geq 1)$  being the $k$th iterative solution to \eqref{generalhpdqsystem}, $\mathbf{w}^{(0)}$ being the initial guess and $\mathbf{w}$ being the exact solution  to \eqref{generalhpdqsystem}. It is known that when applying the PCG solver to linear system of equations in complex field, the step parameters are generated by the inner product of vectors. Similarly,  
																the step parameters of the PCG in quaternion algebra are also given by the inner product of quaternion vectors. Similar to the algorithm in complex filed, the solution of $k$-th iteration 
																${\bf w}^{(k)}$ is given by
																\begin{equation*}
																	\mathbf{w}^{(k)}=\mathbf{w}^{(0)}+\sum\limits^{k-1}_{s=0} \zeta_s(
																	{P^{-1}A})^sP^{-1}(\mathbf{z}-A\mathbf{w}^{(0)}),
																\end{equation*}
															\textcolor{black}{	where the coefficients $\zeta_s$ ($s=0,1,...$) are generated by Hermitian quadratic
																forms / inner products (e.g., ${\bf r}_k^*P^{-1}{\bf r}_k$ and ${\bf p}_k^*A{\bf p}_k$) with ${\bf r}_k$ and ${\bf p}_k$ being some  vectors computed in the iteration. As $P^{-1}$ and $A$ are both quaternionic Hermitian matrices,   $\zeta_s$ ($s=0,1,...$) must be real numbers.} Therefore,  we can estimate the error of the PCG solver in the following theorem.
																
																\begin{theorem}\label{qhpdpcgcvtthm}
																	\begin{description}
																		\item[(i)]
																		The PCG solver admits the iterative error estimation:
																		\begin{equation*}
																			||\mathbf{e}^{(k)}||_{A}\leq ||\mathbf{e}^{(0)}||_{A}\min\limits_{p_k\in\pi_k^1}\max\limits_{\lambda\in\sigma(P^{-1}A)}|p_k(\lambda)|,
																		\end{equation*}
																		where $\mathbf{e}^{(k)}:=\mathbf{w}^{(k)}-\mathbf{w}$ with $\mathbf{w}^{(k)}~(k\geq 1)$  being the $k$th iterative solution to \eqref{generalhpdqsystem}, $\mathbf{w}^{(0)}$ being the initial guess and $\mathbf{w}$ being the exact solution  to \eqref{generalhpdqsystem}. $p_k$ denotes polynomials of degree not larger than $k$; $\pi_k^1$ is the set of polynomials of degree not larger than $k$ such that $p_k(0)=1$.
																		\item[(ii)]
																		The PCG solver can find the exact solution withing at most $n$ iterations.
																	\end{description}
																\end{theorem}
																\begin{proof}
																	Note that $\sigma(P^{-1}A)=\sigma(P^{-1/2}AP^{-1/2})$. Since $P^{-1/2}AP^{-1/2}$ is Hermitian, Lemma \ref{hpdqmatproplm} \bfrefnumber{ii} implies that $|\sigma(P^{-1/2}AP^{-1/2})|\leq n$, where $|\cdot|$ denotes the cardinality of a set. Then, the proof of Theorem \ref{qhpdpcgcvtthm} is similar to the proof of convergence of PCG algorithm for complex Hermitian positive definite system, see, e.g., \cite{axelsson1986rate}. We skip the proof here.
																\end{proof}
																
																\begin{theorem}\label{preparedsuplinearpcgthm}
																	
																	Let $\{\mu_s\}_{s=1}^{n}$ (counting multiplicity) in $\sigma(P^{-1}A)$ be ordered as 
																	\begin{equation*}
																		0<\mu_1\leq \cdots\leq \mu_{l_1}\leq b_1\leq \mu_{l_1+1}\leq \cdots\leq \mu_{n-l_2}\leq b_2\leq \mu_{n-l_2+1}\leq\cdots\leq\mu_n.
																	\end{equation*}
																	Then, the PCG solver admits the following iterative error estimation
																	\begin{equation*}
																		||\mathbf{e}^{(k)}||_{A}\leq 2||\mathbf{e}^{(0)}||_{A}\left(\frac{\sqrt{(b_2/b_1)}-1}{\sqrt{(b_2/b_1)}+1}\right)^{k-l_1-l_2}\prod\limits_{s=1}^{l_1}\left(\frac{b_2-\mu_s}{\mu_s}\right),\quad k\geq l_1+l_2,
																	\end{equation*}
																	where $\mathbf{e}^{(k)}$'s  are defined in Theorem \ref{qhpdpcgcvtthm}.
																\end{theorem}
																
																\begin{proof}
																	The proof is based on estimation of $\min\limits_{p_k\in\pi_k^1}\max\limits_{\lambda\in\sigma(P^{-1}A)}|p_k(\lambda)|$, which can be found in \cite{axelsson1986rate}.
																\end{proof}

																\begin{theorem}\label{strangcircmainthm}\textnormal{(super linear convergence)}
																	Suppose $\check{f}$ is a positive function. Then, for any $\epsilon\in(0,1)$, there exists a constant $C(\epsilon)>0$ and $n_0>0$ such that for  all $n>n_0$, the PCG solver with Strang's circulant preconditioner $c(T_n)$for solving the HPD
																	quaternion system \eqref{hpdqtsystem} admits the following iterative error estimation
																	\begin{equation*}
																		||e^{(k)}||_{T_n}\leq C(\epsilon)\epsilon^{k-4n_0}||e^{(0)}||_{T_n},\quad k\geq 4n_0,
																	\end{equation*}
																	where $\mathbf{e}^{(k)}:=\mathbf{u}^{(k)}-\mathbf{u}$ with $\mathbf{u}^{(k)}~(k\geq 1)$ being the $k$-th iterative solution to \eqref{hpdqtsystem}, $\mathbf{u}^{(0)}$ being the initial guess and $\mathbf{u}$ being the exact solution  to \eqref{generalhpdqsystem}.
																\end{theorem}
																\begin{proof}
																	See Appendix \ref{proofofcircspcvgmainthm}.
																\end{proof}
																
																\section{Numerical Experiments}
																	In this section, we test examples \ref{ar1} and \ref{ar2} arising from quaternion signal processing as outlined in Subsection \ref{qsignalsubsec} by the PCG algorithm.  
																	In practice, $\eta(\cdot)$ is usually unknown, which means the evaluation of \textcolor{black}{$T_n$} and $\mathbf{w}$ given in \eqref{hermqtsystem} may not be practical. In general, samplings of a discrete-time signal are available, which is useful in approximate evaluation of $T_n$ and \textcolor{black}{$\mathbf{w}$}. There are several types of well-known windowing methods with samplings of the signal, such as, correlation, covariance, pre-windowed and post-windowed methods, see, e.g., \cite{ng1994fastiterative,ng2004iterative}. Let $\mathbf{x}_1,\mathbf{x}_2,...,\mathbf{x}_n,\mathbf{x}_{n+1},...,\mathbf{x}_{M}$ be $M$ ($M>n$) many successive samplings of the signal $x(t)$. Now that value of $\mathcal{E}\left(x(t)\overline{x(t+s)}\right)=\eta(s)$ is independent of $t$. Following the correlation windowing method, we use the following sample mean $\tilde{\eta}(s)$ to approximate the expectation  value $\eta(s)$:
																	\begin{equation*}
																		\eta(s)\approx\tilde{\eta}(s):=\frac{1}{M}\sum\limits_{l=1}^{M-s}\mathbf{x}_l\bar{\mathbf{x}}_{l+s},\quad s=0,1,...,n.
																	\end{equation*}
																	{\color{black} With the computed $\tilde{\eta}(s)$ $(s=0,1,...,n)$, we obtain another Hermitian quaternion Toeplitz linear system \eqref{apphermqtsystem} to approximate  \eqref{hermqtsystem},}
																	\begin{equation}\label{apphermqtsystem}
																		\tilde{H}\tilde{\mathbf{v}}=\tilde{\mathbf{w}},
																	\end{equation}
																	where
																	\begin{align*}
																		&\tilde{H}:=\left[
																		\begin{array}
																			[c]{ccccc}
																			\tilde{\eta}(0)  &  \overline{\tilde{\eta}(1)} &\ldots   & \overline{\tilde{\eta}(n-2)}   & \overline{\tilde{\eta}(n-1)} \\
																			\tilde{\eta}(1 ) &  \tilde{\eta}(0) & \overline{\tilde{\eta}(1)}    &\ddots & \overline{\tilde{\eta}(n-2)} \\
																			\vdots&\ddots &\ddots&\ddots&\vdots\\
																			\tilde{\eta}(n-2)  & \ddots& \tilde{\eta}(1) &  \tilde{\eta}(0)  & \overline{\tilde{\eta}(1)}  \\
																			\tilde{\eta}(n-1)  &  \tilde{\eta}(n-2) & \ldots &  \tilde{\eta}(1)  &  \tilde{\eta}(0) 
																		\end{array}
																		\right],\\
																		&\tilde{\mathbf{w}}:=\left[\tilde{\eta}(1),\tilde{\eta}(2),..., \tilde{\eta}(n)\right]^{\rm T}.
																	\end{align*}
																	On the other hand, $\tilde{H}$ can be rewritten as $\tilde{H}=\frac{1}{M}\tilde{T}^{*}\tilde{T}$ with 
																	\begin{equation*}
																		\tilde{T}:=\left[\begin{array}[c]{ccc}
																			\bar{x}_{1}&&\\
																			\vdots&\ddots&\\
																			\bar{x}_{n}&\ldots&\bar{x}_{1}\\
																			\vdots&\ddots&\vdots\\
																			\vdots&\ddots&\vdots\\
																			\bar{x}_{M}&\ldots&\bar{x}_{M-n+1}\\
																			&\ddots&\vdots\\
																			&&\bar{x}_{M}
																		\end{array}\right]\in\mathbb{Q}^{(M+n-1)\times n}.
																	\end{equation*}
																	Hence, if $\tilde{T}$ is a full rank matrix, then $\tilde{H}$ is an HPD quaternion matrix. One can then apply the PCG  solver with the circulant preconditioner to solve the linear system \eqref{apphermqtsystem}.
																	
																	{\bf Settings:} In the experiments, zero initial guess is taken for the PCG algorithm; the stopping criterion of PCG algorithm is 
																	set as $||\hat{\mathbf{r}}^{(k)}||_2\leq 10^{-7}||\hat{\mathbf{r}}^{(0)}||_2$, where $\hat{\mathbf{r}}^{(k)}$ denotes the residual vector at $k$-th PCG iteration and $\hat{\mathbf{r}}_0$ denotes the initial residual vector. 
																	To demonstrate the effectiveness of Strang's circulant preconditioner on quaternionic Toeplitz systems, we test the PCG algorithm with the Strang's circulant preconditioner and unpreconditioned conjugate gradient algorithm in the lateral experiments and compare their performance. For ease of statement, we use  PCG-${\bf C}$ (CG, resp.) to represent the PCG algorithm with Strang's circulant preconditioner (unpreconditioned conjugate gradient solver, resp.).  We shall use `CPU' to represent the computational time; and  `Iter' for the iteration numbers of PCG solvers. To quantify the accuracy of PCG solvers, we  define error measure of iterative solution as
																	\begin{equation*}
																		{\rm Error}:=||\mathbf{b}-T_n\mathbf{u}^{(k)}||_{2},
																	\end{equation*}
																	where $\mathbf{u}^{(k)}$ denotes some iterative solution to \eqref{hpdqtsystem} or \eqref{apphermqtsystem}.
																	
																		In the following, we present the results of  Examples \ref{ar1} and \ref{ar2} by applying the PCG solvers and list the results in Tables \ref{ar1exacttable}-\ref{ar2apptable}. From these tables, we see that  (i) PCG-${\bf C}$ is more efficient than CG in terms of iteration number and computational time; (ii) PCG-${\bf C}$ is  more accurate than CG in terms of Error. 
																		The better performance of PCG-${\bf C}$ compared with that of CG demonstrates the effectiveness of the circulant preconditioner.
																		Such performance of circulant preconditioner for quaternion Hermitian Toeplitz system is consistent with what we have usually observed on circulant preconditioning for complex Hermitian Toeplitz system.
																		\begin{table}[h]
																			\caption{Numerical results of  iterative solutions of \eqref{hermqtsystem} arising from Example \ref{ar1}}\label{ar1exacttable}
																			\renewcommand{\arraystretch}{1.1}
																			\setlength{\tabcolsep}{0.7em}
																				\hspace{-6mm}
																				\begin{tabular}{c|c|ccc|ccc}
																					\hline
																					\multirow{2}{*}{$\beta$}&\multirow{2}{*}{$n$} &\multicolumn{3}{c|}{PCG-${\bf C}$}&\multicolumn{3}{c}{CG} \\    
																					\cline{3-8}
																					&&$\mathrm{Iter}$&$\mathrm{CPU}$&Error&$\mathrm{Iter}$&$\mathrm{CPU}$&Error\\
																					\hline													
																					\multirow{4}{*}{0.45-0.01\unitp+0.3\unitq-0.35\unitr}		
																					& $2^8$	&	3	&	0.075	&	8.83e-15	&	41	&	0.219 &	9.35e-8	\\	
																					& $2^9$	&	3	&	0.127	&	9.77e-15	&	41	&	0.225 &	9.35e-8	\\	
																					& $2^{10}$	&	3	&	0.250	&	7.89e-15	&	41	&0.351 &	9.35e-8	\\	
																					& $2^{11}$	& 3	&	0.481	&	8.76e-15	&	41	&0.628	 &	9.35e-8	\\	
																					\hline
																					\multirow{4}{*}{-0.07+0.41\unitp+0.29\unitq+0.45\unitr}	
																					& $2^8$	&	3	&	0.068	&	1.98e-14	&	48	&	0.176 &	7.78e-8	\\	
																					& $2^9$	&	3	&	0.127	&	2.23e-14	&	48	&	0.228 &	7.83e-8	\\	
																					& $2^{10}$	&	3	&	0.241	&	2.38e-14	&	48	&0.354 &	7.83e-8	\\	
																					& $2^{11}$	& 3	&	0.490	&	2.05e-14	&	48	&0.634	 &	7.83e-8	\\	
																					\hline
																					\multirow{4}{*}{0.15-0.46\unitp+0.34\unitq+0.43\unitr}	
																					& $2^8$	&	3	&	0.069	&	5.44e-14	&	57	&	0.177 &	6.55e-8	\\	
																					& $2^9$	&	3	&	0.124	&	5.57e-14	&	60	&	0.251 &	9.61e-8	\\	
																					& $2^{10}$	&	3	&	0.244	&	5.20e-14	&	60	&0.374 &	9.61e-8	\\	
																					& $2^{11}$	& 3	&	0.485	&	5.78e-14	&	60	&0.657	 &	9.61e-8	\\	
																					\hline
																				\end{tabular}
																		\end{table}

																		\begin{table}[h]
																			\caption{Numerical results of  iterative solutions of \eqref{apphermqtsystem} arising from Example \ref{ar1} with $M=mn+1$}\label{ar1apptable}
																			\renewcommand{\arraystretch}{1.1}
																			\setlength{\tabcolsep}{0.7em}
																			\begin{center}
																				\begin{tabular}{c|c|c|ccc|ccc}
																					\hline
																					\multirow{2}{*}{$\beta$}&\multirow{2}{*}{$m$}&\multirow{2}{*}{$n$} &\multicolumn{3}{c|}{PCG-${\bf C}$}&\multicolumn{3}{c}{CG} \\    
																					\cline{4-9}
																					&&&$\mathrm{Iter}$&$\mathrm{CPU}$&Error&$\mathrm{Iter}$&$\mathrm{CPU}$&Error\\
																					\hline													
																					\multirow{9}{*}{0.1+0\unitp-0.3\unitq-0.4\unitr}		
																					&\multirow{3}{*}{$2^2$}& $2^9$	&	27	&	0.183	&	6.85e-8	&	53	&	0.243 &	7.57e-8	\\	
																					&& $2^{10}$	&	29	&	0.332	&	7.05e-8	&	61	&	0.461 &	7.72e-8	\\	
																					&& $2^{11}$	&	41	&	0.648	&	8.89e-8	&	63	&0.839 &	7.36e-8	\\	
																					\cline{2-9}
																					&\multirow{3}{*}{$2^3$}& $2^{9}$	& 17	&	0.174	&	9.57e-8	&	45	&0.216	 &	9.67e-8	\\	
																					&& $2^{10}$	&	19	&	0.296	&	4.31e-8	&	48	&	0.353 &	9.53e-8	\\	
																					&& $2^{11}$	&	20	&	0.551	&	4.77e-8	&	51	&0.652 &	7.50e-8	\\	
																					\cline{2-9}
																					&\multirow{3}{*}{$2^4$}& $2^{9}$	& 13	&	0.146	&	9.54e-8	&	36	&0.194	 &	9.95e-8	\\	
																					&& $2^{10}$	&	14	&	0.280	&	4.50e-8	&	39	&	0.350 &	6.89e-8	\\	
																					&& $2^{11}$	&	14	&	0.521	&	5.22e-8	&	41	&0.597 &	7.58e-8	\\	
																					\hline
																					\multirow{9}{*}{0.3+0.4\unitp+0\unitq+0.4\unitr}		
																					&\multirow{3}{*}{$2^2$}& $2^9$	&	18	&	0.166	&	8.75e-8	&	64	&	0.257 &	7.54e-8	\\	
																					&& $2^{10}$	&	19	&	0.287	&	8.28e-8	&	67	&	0.394 &	7.82e-8	\\	
																					&& $2^{11}$	&	20	&	0.534	&	5.86e-8	&	68	&0.684 &	9.00e-8	\\	
																					\cline{2-9}
																					&\multirow{3}{*}{$2^3$}& $2^{9}$	& 14	&	0.146	&	8.46e-8	&	55	&0.229	 &	7.12e-8	\\	
																					&& $2^{10}$	&	14	&	0.263	&	5.85e-8	&	59	&	0.367 &	9.07e-8	\\	
																					&& $2^{11}$	&	15	&	0.547	&	3.29e-8	&	60	&0.653 &	7.52e-8	\\	
																					\cline{2-9}
																					&\multirow{3}{*}{$2^4$}& $2^{9}$	& 14	&	0.154	&	4.50e-8	&	54	&0.244	 &	8.96e-8	\\	
																					&& $2^{10}$	&	15	&	0.278	&	3.10e-8	&	56	&	0.380 &	7.50e-8	\\	
																					&& $2^{11}$	&	15	&	0.522	&	4.33e-8	&	64	&0.671 &	9.57e-8	\\	
																					\hline
																					\multirow{9}{*}{0.3+0.4\unitp+0.4\unitq+0\unitr}		
																					&\multirow{3}{*}{$2^2$}& $2^9$	&	28	&	0.175	&	8.46e-8	&	78	&	0.280 &	9.94e-8	\\	
																					&& $2^{10}$	&	28	&	0.306	&	4.62e-8	&	86	&	0.453 &	8.81e-8	\\	
																					&& $2^{11}$	&	42	&	0.607	&	7.87e-8	&	99	&0.812 &	8.87e-8	\\	
																					\cline{2-9}
																					&\multirow{3}{*}{$2^3$}& $2^{9}$	& 18	&	0.159	&	5.59e-8	&	66	&0.257	 &	8.71e-8	\\	
																					&& $2^{10}$	&	20	&	0.281	&	7.73e-8	&	68	&	0.392 &	7.49e-8	\\	
																					&& $2^{11}$	&	21	&	0.545	&	6.93e-8	&	71	&0.694 &	9.57e-8	\\	
																					\cline{2-9}
																					&\multirow{3}{*}{$2^4$}& $2^{9}$	& 14	&	0.152	&	6.08e-8	&	53	&0.232	 &	9.41-8	\\	
																					&& $2^{10}$	&	15	&	0.288	&	2.38e-8	&	58	&	0.382 &	7.33e-8	\\	
																					&& $2^{11}$	&	15	&	0.527	&	5.21e-8	&	63	&0.667 &	8.90e-8	\\	
																					\hline
																				\end{tabular}
																			\end{center}
																		\end{table}	
																	}

																	\begin{table}[h]
																		\caption{Numerical results of  iterative solutions of \eqref{hermqtsystem} arising from Example \ref{ar2}}\label{ar2exacttable}
																		\renewcommand{\arraystretch}{1.1}
																		\setlength{\tabcolsep}{0.7em}
																		\begin{center}
																			\begin{tabular}{c|c|ccc|ccc}
																				\hline
																				\multirow{2}{*}{$\beta$}&\multirow{2}{*}{$n$} &\multicolumn{3}{c|}{PCG-${\bf C}$}&\multicolumn{3}{c}{CG} \\    
																				\cline{3-8}
																				&&$\mathrm{Iter}$&$\mathrm{CPU}$&Error&$\mathrm{Iter}$&$\mathrm{CPU}$&Error\\
																				\hline													
																				\multirow{4}{*}{-0.08+0.21\unitp-0.8\unitq-0.79\unitr}		
																				& $2^8$	&	2	&	0.096	&	5.73e-13	&	119	&	0.339 &	9.38e-8	\\	
																				& $2^9$	&	2	&	0.131	&	6.07e-13	&119	&	0.402 &	9.38e-8	\\	
																				& $2^{10}$	&	2	&	0.254	&	6.22e-13	&	119	&0.545 &	9.38e-8	\\	
																				& $2^{11}$	& 2	&	0.509	&	6.29e-13	&	119	&0.871	 &	9.38e-8	\\	
																				\hline
																				\multirow{4}{*}{-0.2+0.18\unitp-1.19\unitq-0.07\unitr}	
																				& $2^8$	&	2	&	0.068	&	1.44e-13	&	83	&	0.233 &	9.55e-8	\\	
																				& $2^9$	&	2	&	0.126	&	1.90e-13	&	83	&	0.302 &	9.55e-8	\\	
																				& $2^{10}$	&	2	&	0.246	&	1.80e-13	&	83	&0.441 &	9.55e-8	\\	
																				& $2^{11}$	& 2	&	0.497	&	1.88e-13	&	83	&0.744	 &	9.55e-8	\\	
																				\hline
																				\multirow{4}{*}{-0.52-0.32\unitp-0.01\unitq-1.23\unitr}	
																				& $2^8$	&	2	&	0.066	&	2.57e-14	&	54	&	0.172 &	9.40e-8	\\	
																				& $2^9$	&	2	&	0.124	&	2.40e-14	&	54	&	0.235 &	9.40e-8	\\	
																				& $2^{10}$	&	2	&	0.243	&	2.44e-14	&	54	&0.367 &	9.40e-8	\\	
																				& $2^{11}$	& 2	&	0.488	&	2.72e-14	&	54	&0.650	 &	9.40e-8	\\	
																				\hline
																			\end{tabular}
																		\end{center}
																	\end{table}

																	\begin{table}[h]
																		\caption{Numerical results of  iterative solutions of \eqref{apphermqtsystem} arising from Example \ref{ar2} with $M=mn+1$}\label{ar2apptable}
																		\renewcommand{\arraystretch}{1.1}
																		\setlength{\tabcolsep}{0.7em}
																		\hspace{-6mm}
																		\begin{tabular}{c|c|c|ccc|ccc}
																			\hline
																			\multirow{2}{*}{$\beta$}&\multirow{2}{*}{$m$}&\multirow{2}{*}{$n$} &\multicolumn{3}{c|}{PCG-${\bf C}$}&\multicolumn{3}{c}{CG} \\    
																			\cline{4-9}
																			&&&$\mathrm{Iter}$&$\mathrm{CPU}$&Error&$\mathrm{Iter}$&$\mathrm{CPU}$&Error\\
																			\hline													
																			\multirow{9}{*}{0.9+0.9\unitp+0.5\unitq+1.3\unitr}		
																			&\multirow{3}{*}{$2^2$}& $2^9$	&	36	&	0.196	&	5.04e-8	&	64	&	0.258 &	9.50e-8	\\	
																			&& $2^{10}$	&	33	&	0.318	&	9.13e-8	&	64	&	0.392 &	7.55e-8	\\	
																			&& $2^{11}$	&	46	&	0.622	&	5.92e-8	&	76	&0.715 &	7.77e-8	\\	
																			\cline{2-9}
																			&\multirow{3}{*}{$2^3$}& $2^{9}$	& 19	&	0.162	&	9.51e-8	&	50	&0.227	 &	9.90e-8	\\	
																			&& $2^{10}$	&	18	&	0.283	&	7.19e-8	&	50	&	0.362 &	6.85e-8	\\	
																			&& $2^{11}$	&	19	&	0.545	&	7.23e-8	&	57	&0.668 &	7.60e-8	\\	
																			\cline{2-9}
																			&\multirow{3}{*}{$2^4$}& $2^{9}$	& 13	&	0.147	&	9.94e-8	&	41	&0.208	 &	8.28e-8	\\	
																			&& $2^{10}$	&	14	&	0.278	&	5.66e-8	&	44	&	0.344 &	7.44e-8	\\	
																			&& $2^{11}$	&	14	&	0.533	&	4.79e-8	&	43	&0.617 &	6.96e-8	\\	
																			\hline
																			\multirow{9}{*}{-1.9-0.6\unitp+0.3\unitq+0\unitr}		
																			&\multirow{3}{*}{$2^2$}& $2^9$	&	29	&	0.191	&	7.99e-8	&	67	&	0.259 &	9.79e-8	\\	
																			&& $2^{10}$	&	37	&	0.330	&	6.92e-8	&	69	&	0.403 &	7.98e-8	\\	
																			&& $2^{11}$	&	38	&	0.600	&	8.56e-8	&	69	&0.697 &	8.68e-8	\\	
																			\cline{2-9}
																			&\multirow{3}{*}{$2^3$}& $2^{9}$	& 17	&	0.160	&	7.74e-8	&	46	&0.224	 &	8.46e-8	\\	
																			&& $2^{10}$	&	18	&	0.284	&	8.83e-8	&	48	&	0.356 &	7.98e-8	\\	
																			&& $2^{11}$	&	22	&	0.561	&	5.24e-8	&	53	&0.633 &	8.15e-8	\\	
																			\cline{2-9}
																			&\multirow{3}{*}{$2^4$}& $2^{9}$	& 14	&	0.147	&	5.56e-8	&	40	&0.204	 &	7.81e-8	\\	
																			&& $2^{10}$	&	15	&	0.267	&	3.36e-8	&	44	&	0.343 &	7.01e-8	\\	
																			&& $2^{11}$	&	14	&	0.516	&	6.83e-8	&	43	&0.601 &	6.98e-8	\\	
																			\hline
																			\multirow{9}{*}{-2-0.6\unitp-0.4\unitq-0.1\unitr}		
																			&\multirow{3}{*}{$2^2$}& $2^9$	&	25	&	0.172	&	7.70e-8	&	61	&	0.250 &	6.90e-8	\\	
																			&& $2^{10}$	&	32	&	0.321	&	6.39e-8	&	65	&	0.392 &	7.08e-8	\\	
																			&& $2^{11}$	&	46	&	0.623	&	9.54e-8	&	69	&0.693 &	8.30e-8	\\	
																			\cline{2-9}
																			&\multirow{3}{*}{$2^3$}& $2^{9}$	& 19	&	0.166	&	4.04e-8	&	43	&0.217	 &	9.51e-8	\\	
																			&& $2^{10}$	&	19	&	0.290	&	4.42e-8	&	44	&	0.363 &	8.82e-8	\\	
																			&& $2^{11}$	&	19	&	0.546	&	4.78e-8	&	47	&0.623 &	9.90e-8	\\	
																			\cline{2-9}
																			&\multirow{3}{*}{$2^4$}& $2^{9}$	& 13	&	0.145	&	5.99e-8	&	38	&0.202	 &	7.85-8	\\	
																			&& $2^{10}$	&	14	&	0.268	&	5.85e-8	&	39	&	0.323 &	6.51e-8	\\	
																			&& $2^{11}$	&	15	&	0.529	&	2.80e-8	&	40	&0.595 &	9.94e-8	\\	
																			\hline
																		\end{tabular}
																	\end{table}


																	\section{Concluding Remarks}
																	In this paper, 
																	we have studied a sequence of Hermitian quaternion Toeplitz matrices generated by a quaternion-valued function 
																	which is the sum of a real-valued function and an odd function with imaginary component. This is different from the case 
																	of Hermitian complex Toeplitz matrices generated by real-valued functions only. As an example, 
																	we studied the Strang circulant preconditioner for quaternion Hermitian Toeplitz matrix. We have shown that 
																	the Strang circulant preconditioner can be diagonalized by discrete Fourier transform matrix whereas a general quaternion
																	circulant matrix cannot be diagonized. Both theoretical and numerical results are presented to demonstrate 
																	the effectiveness of quaternion circulant preconditioner for solving quaternion Toeplitz system.
																	
																	It is interesting to note that the spectra of complex Toeplitz matrices are characterized by 
																	their complex generating functions on the unit circle
																	in the complex plane and their winding numbers \cite{reichel1992eigenvalues}.  
																	As a future research work, it is worth how the spectra of 
																	quaternion Toeplitz matrices are related to their quaternion generating functions and the possible 
																	generalization of winding number in the hypercomplex domain. 
						
\section*{Acknowledgments}
The work of Xue-Lei Lin was partially supported by research grants: 12301480 from NSFC, 2025A1515010945 from Natural Science Foundation of Guangdong Province.										
																	\bibliographystyle{siam}
	
																	\begin{appendix}
																		\section{Proof of Theorem \ref{quatszegothm}}\label{thmquatszegothmproof}
																		Note that $\matcomplexify{T_n}$ is similar to the following matrix $B_n$ by a permutation transformation
																		\begin{equation*}
																			B_n:=\left[B_n^{(s,l)}\right]_{s,l=1}^{n},
																		\end{equation*}
																		with
																		\begin{equation*}
																			B_n^{(s,l)}=\left[\begin{array}[c]{cc}
																				[\phi_1(T_n)](s,l)&[-\phi_2(T_n)](s,l)\\
																				&\\
																				\Big[\overline{\phi_2(T_n)}\Big](s,l)&\Big[\overline{\phi_1(T_n)}\Big](s,l)
																			\end{array}\right],\quad s,l=1,2,...,n.
																		\end{equation*}
																		Then, matrix similarity implies that $\matcomplexify{T_n}$ and $B_n$ have the same set of eigenvalues counting the multiplicity. We assume that their eigenvalues are  in ascending order, specifically that $\lambda_1(\cdot)\leq \lambda_2(\cdot)\leq \cdots \leq \lambda_m(\cdot) $. Moreover, according to results in \cite{miranda2000asymptotic}, we know that if $f\in L^{\infty}([-\pi,\pi])$, then for any continuous function $F$ on the closed interval $[\check{a},\hat{a}]$, it also holds that
																		\begin{equation}\label{linfspectrdisctrieq}
																			\lim\limits_{n\rightarrow\infty}\frac{1}{2n}\sum\limits_{s=1}^{2n}F(\lambda_{s}(B_n))=\frac{1}{4\pi}\int_{-\pi}^{\pi}\sum\limits_{s=1}^{2}F\left(\lambda_{s}\left(G[f](x)\right)\right)dx;
																		\end{equation}
																		and that if $f\in L^{2}([-\pi,\pi])$, then for any continuous function $F$ with compact support in $\mathbb{R}$, it holds that
																		\begin{equation}\label{l2spectrdisctrieq}
																			\lim\limits_{n\rightarrow\infty}\frac{1}{2n}\sum\limits_{s=1}^{2n}F(\lambda_{s}(B_n))=\frac{1}{4\pi}\int_{-\pi}^{\pi}\sum\limits_{s=1}^{2}F\left(\lambda_{s}\left(G[f](x)\right)\right)dx.
																		\end{equation}
																		Here, $G[f](\cdot)$ is defined in Theorem \ref{qhermitmatrayleightogfuncthm}\bfrefnumber{i}.
																		
																		Let $T_n=U^{*}DU$ be unitary diagonalization of $T_n$ with $U$ being an unitary matrix and $D$ being a real diagonal matrix. Then, it is straightforward to see that 
																		\begin{equation*}
																			\matcomplexify{T_n}=\matcomplexify{U}^{*}\matcomplexify{D}\matcomplexify{U}.
																		\end{equation*} 
																		Note that $\matcomplexify{U}$ is also unitary and that $\matcomplexify{D}={\rm diag}(D,D)$ is a diagonal matrix. Thus, 
																		\begin{equation*}
																			\lambda_{2s-1}\left(\matcomplexify{T_n}\right)=\lambda_{2s}\left(\matcomplexify{T_n}\right)=\lambda_{s}(T_n),\quad s=1,2,...,n.
																		\end{equation*}
																		Then, for $F$ appears in \eqref{linfspectrdisctrieq} or \eqref{l2spectrdisctrieq}, it holds that
																		\begin{equation}\label{lefthandsideeqform}
																			\sum\limits_{s=1}^{2n}F(\lambda_{s}(B_n))=\sum\limits_{s=1}^{2n}F(\lambda_{s}\left(\matcomplexify{T_n}\right))=2\sum\limits_{s=1}^{n}F(\lambda_{s}(T_n)),
																		\end{equation}
																		where the first equality comes from the fact that $\matcomplexify{T_n}$ is similar to $B_n$.
																		
																		On the other hand, from the proof of Theorem \ref{qhermitmatrayleightogfuncthm}, we see that
																		\begin{equation*}
																			\lambda_1\left(G[f](x)\right)=\check{f}(x),\quad 	\lambda_2\left(G[f](x)\right)=\hat{f}(x),
																		\end{equation*}
																		which together with \eqref{linfspectrdisctrieq}, \eqref{l2spectrdisctrieq} and \eqref{lefthandsideeqform} implies that
																		\begin{equation*}
																			\lim\limits_{n\rightarrow\infty}\frac{1}{n}\sum\limits_{s=1}^{n}F(\lambda_{s}(T_n))=\frac{1}{4\pi}\int_{-\pi}^{\pi}\left[F(\hat{f}(x))+F(\check{f}(x))\right]dx,
																		\end{equation*}
																		holds for $F$ appearing in \eqref{linfspectrdisctrieq} or \eqref{l2spectrdisctrieq}.
																		The proof is complete.
																		
																		\section{Proof of Lemma \ref{tnminuscstneigcontrllm}}\label{lemmtnmiusproof}
																		By proof of Lemma \ref{strangcireiglimitlm}, we have
																		\begin{equation*}
																			\matcomplexify{T_n}-\matcomplexify{c(T_n)}=\left[\begin{array}[c]{cc}
																				\underbrace{\phi_1(T_n)-	c(\phi_1(T_n))}_{:=B_{n}^{(1,1)}}&\underbrace{c(\phi_2(T_n))-\phi_2(T_n)}_{:=B_{n}^{(1,2)}}\\
																				&\\
																				\underbrace{\overline{\phi_2(T_n)}-	c\left(\overline{\phi_2(T_n)}\right)}_{:=B_{n}^{(2,1)}}&\underbrace{\overline{\phi_1(T_n)}-c\left(\overline{\phi_1(T_n)}\right)}_{:=B_{n}^{(2,2)}}
																			\end{array}\right].
																		\end{equation*}
																		Since $\{t_s\}_{s\in\mathbb{N}}$ is absolutely summable, $\lim\limits_{n\rightarrow\infty}\sum\limits_{l>n}|t_l|=0$. That means for any $\epsilon>0$, there exists $n_0>0$ such that $ \sum\limits_{l\geq n_0}|t_l|\leq \frac{\epsilon}{2}$.

																		Let $U_{(n,n_0)}^{(s,l)}$ be the $n\times n$ matrix obtained from $B_{n}^{(s,l)}$ by replacing the $(n-n_0)$-by-$(n-n_0)$ leading principal sub-matrix of $B_{n}^{(s,l)}$ with zero matrix for $s,l=1,2$. Then, it is clear that
																		$\max\limits_{s,l\in\{1,2\}}{\rm rank}(U_{(n,n_0)}^{(s,l)})\leq 2n_0$.
																		
																		Note that \bfrefnumber{i} the nonzero entries of $B_n^{(1,1)}-U_{(n,n_0)}^{(1,1)}$ are all located in its $(n-n_0)$-by-$(n-n_0)$ leading principal sub-matrix; \bfrefnumber{ii} the central $2\lfloor (n-1)/2\rfloor+1$ diagonals of $(n-n_0)$-by-$(n-n_0)$ leading principal sub-matrix of $B_n^{(1,1)}-U_{(n,n_0)}^{(1,1)}$ are all zeros; \bfrefnumber{iii} $(n-n_0)$-by-$(n-n_0)$ leading principal sub-matrix of $B_n^{(1,1)}-U_{(n,n_0)}^{(1,1)}$ is a Toeplitz matrix. It is easy to see from these facts that $||B_n^{(1,1)}-U_{(n,n_0)}^{(1,1)}||_{1}$ is  equal to $1$-norm of first column of  $(n-n_0)$-by-$(n-n_0)$ leading principal sub-matrix of $B_n^{(1,1)}-U_{(n,n_0)}^{(1,1)}$. That means
																		\begin{align*}
																			&||B_n^{(1,1)}-U_{(n,n_0)}^{(1,1)}||_{1}\\
																			&=\sum\limits_{s=\lfloor (n-1)/2\rfloor+2}^{n-n_0}\left|[B_n^{(1,1)}-U_{(n,n_0)}^{(1,1)}](s,1)\right|=\sum\limits_{s=\lfloor (n-1)/2\rfloor+2}^{n-n_0}\left|(B_n^{(1,1)})(s,1)\right|\\
																			&=\sum\limits_{s=\lfloor (n-1)/2\rfloor+2}^{n-n_0}|[\phi_1(T_n)](s,1)-[c(\phi_1(T_n))](s,1)|\\
																			&\leq \sum\limits_{s=\lfloor (n-1)/2\rfloor+2}^{n-n_0}|t_s^{(0)}+t_s^{(1)}\unitp|+\sum\limits_{s=n_0}^{\lfloor (n-1)/2\rfloor}|t_s^{(0)}-t_s^{(1)}\unitp|\leq \sum\limits_{s=n_0}^{n-n_0}|t_s^{(0)}|+|t_s^{(1)}|.
																		\end{align*}
																		{\color{black} Similarly, one can show that		
																			\begin{align*}
																				&||B_n^{(2,2)}-U_{(n,n_0)}^{(2,2)}||_{1}\leq \sum\limits_{s=n_0}^{n-n_0}|t_s^{(0)}|+|t_s^{(1)}|,\quad 
																				||B_n^{(1,2)}-U_{(n,n_0)}^{(1,2)}||_{1}\leq \sum\limits_{s=n_0}^{n-n_0}|t_s^{(2)}|+|t_s^{(3)}|,\\
																				&||B_n^{(2,1)}-U_{(n,n_0)}^{(2,1)}||_{1}\leq \sum\limits_{s=n_0}^{n-n_0}|t_s^{(2)}|+|t_s^{(3)}|.
																		\end{align*}}
																		Denote
																		\begin{equation*}
																			W_n=\left[\begin{array}[c]{cc}
																				B_n^{(1,1)}-U_{(n,n_0)}^{(1,1)}&B_n^{(1,2)}-U_{(n,n_0)}^{(1,2)}\\
																				&\\
																				B_n^{(2,1)}-U_{(n,n_0)}^{(2,1)}&B_n^{(2,2)}-U_{(n,n_0)}^{(2,2)}
																			\end{array}\right],~U_{(n,n_0)}=\left[\begin{array}[c]{cc}
																				U_{(n,n_0)}^{(1,1)}&U_{(n,n_0)}^{(1,2)}\\
																				&\\
																				U_{(n,n_0)}^{(2,1)}&U_{(n,n_0)}^{(2,2)}
																			\end{array}\right].
																		\end{equation*}
																		Then, $\matcomplexify{T_n}-\matcomplexify{c(T_n)}=W_n+U_{(n,n_0)}.$  
																		Note that both $W_n$ and $U_{(n,n_0)}$ are Hermitian  
																		${\rm rank}(U_{(n,n_0)})\leq 4n_0.$

																		{\color{black}	On the other hand,
																			\begin{eqnarray*}
																				&& ||W_n||_2=\rho(W_n)\leq ||W_n||_{1}\\
																				&\leq & \max\Big\{||B_n^{(1,1)}-U_{(n,n_0)}^{(1,1)}||_{1}+||B_n^{(2,1)}-U_{(n,n_0)}^{(2,1)}||_{1},  ~||B_n^{(1,2)}-U_{(n,n_0)}^{(1,2)}||_{1}+||B_n^{(2,2)}-U_{(n,n_0)}^{(2,2)}||_{1}\Big\}\\
																				&\leq & \sum\limits_{s=n_0}^{n-n_0}|t_s^{(0)}|+|t_s^{(1)}|+|t_s^{(2)}|+|t_s^{(3)}| \leq 2\sum\limits_{s=n_0}^{n-n_0}|t_s|\leq 2\sum\limits_{s\geq n_0}|t_s|\leq \epsilon,
																			\end{eqnarray*}
																			where the fourth inequality is from Cauchy Schwartz inequality. The result follows.}

																		\section{Proof of Theorem \ref{strangcircprecedmatspectrumthm}}\label{proofofcircprecedmatspecclstthm}
																		\begin{proof}
																			Since $\check{f}$ is a continuous positive function, there exists $x_0\in[-\pi,\pi]$ such that $\check{f}_{\min}=f(x_0)>0$. Note that $\hat{f}\geq \check{f}$,  we have $\hat{f}_{\max}\geq \check{f}_{\min}>0$.
																			It follows from Lemma  \ref{strangcireiglimitlm} that there exists $n_1>0$ such that for all $n\geq n_1$, it holds that
																			\begin{equation*}
																				\lambda_{\min}(c(T_n))\geq \frac{\check{f}_{\min}}{2}>0{\rm~~and~~}\lambda_{\max}(c(T_n))\leq \frac{3\check{f}_{\max}}{2}
																			\end{equation*}
																			which guarantees that  $c(T_n)$ is HPD. 
																			Then,
																			\begin{align*}
																				\sigma\left(c(T_n)^{-1}T_n\right)&=\sigma\left(c(T_n)^{-1/2}T_nc(T_n)^{-1/2}\right)\\
																				&=\sigma\left(\matcomplexify{c(T_n)^{-1/2}T_nc(T_n)^{-1/2}}\right)\\
																				&=\sigma\left(\matcomplexify{c(T_n)^{-1/2}}\matcomplexify{T_n}\matcomplexify{c(T_n)^{-1/2}}\right)\\
																				&=\sigma\left(\matcomplexify{c(T_n)}^{-1/2}\matcomplexify{T_n}\matcomplexify{c(T_n)}^{-1/2}\right).
																			\end{align*}
																			Note that
																			\begin{align*}
																				&\sigma(\matcomplexify{c(T_n)}^{-1/2}\matcomplexify{T_n}\matcomplexify{c(T_n)}^{-1/2})\\
																				&=\sigma\left(I_{2n}+\matcomplexify{c(T_n)}^{-1/2}[\matcomplexify{T_n}-\matcomplexify{c(T_n)}]\matcomplexify{c(T_n)}^{-1/2}\right).
																			\end{align*}
																			By Lemma \ref{tnminuscstneigcontrllm},  there exists $n_0\geq n_1$ such that for all $n>n_0$, it holds that
																			\begin{equation*}
																				\matcomplexify{T_n}-\matcomplexify{c(T_n)}=W_n+U_{(n,n_0)},\quad ||W_n||_2\leq \frac{\check{f}_{\min}\epsilon}{2},\quad {\rm rank}(U_{(n,n_0)})\leq 4n_0.
																			\end{equation*}
																			Then,
																			\begin{align*}
																				&\matcomplexify{c(T_n)}^{-1/2}[\matcomplexify{T_n}-\matcomplexify{c(T_n)}]\matcomplexify{c(T_n)}^{-1/2}\\
																				&=\underbrace{\matcomplexify{c(T_n)}^{-1/2}W_n\matcomplexify{c(T_n)}^{-1/2}}_{:=\tilde{W}_n}+\underbrace{\matcomplexify{c(T_n)}^{-1/2}U_{(n,n_0)}\matcomplexify{c(T_n)}^{-1/2}}_{:=\tilde{U}_{(n,n_0)}},
																			\end{align*}
																			and for $n >n_0$,
																			\begin{align*}
																				&||\tilde{W}_n||_2\leq ||\matcomplexify{c(T_n)}^{-1/2}||_2^2||W_n||_2\leq \frac{2}{\check{f}_{\min}}\times \frac{\check{f}_{\min}\epsilon}{2}=\epsilon,\\
																				&{\rm rank}(\tilde{U}_{(n,n_0)})={\rm rank}(U_{(n,n_0)})\leq 4n_0.
																			\end{align*}
																			{\color{black} Hence,
																				\begin{align*}
																					\sigma\left(c(T_n)^{-1}T_n\right)&=\sigma\left(I_{2n}+\tilde{W}_n+\tilde{U}_{(n,n_0)}\right) =\left\{1+\lambda\Big|\lambda\in\sigma(\tilde{W}_n+\tilde{U}_{(n,n_0)})\right\}.
																				\end{align*}		
																				Moreover, Weyl's Theorem implies that for $s=1,2,...,2n$,
																				\begin{align}\label{weylneq1}
																					\lambda_1(\tilde{W}_n)+\lambda_s(\tilde{U}_{(n,n_0)}) \leq \lambda_s(\tilde{W}_n+\tilde{U}_{(n,n_0)}) \leq \lambda_{2n}(\tilde{W}_n)+\lambda_s(\tilde{U}_{(n,n_0)}). 
																			\end{align}}
																			By the facts that $\tilde{U}_{(n,n_0)}$ is Hermitian and that ${\rm rank}(\tilde{U}_{(n,n_0)})\leq 4n_0$, we see that $\tilde{U}_{(n,n_0)}$ has at least $2n-4n_0$ zero eigenvalues, which together with \eqref{weylneq1} implies that there are at least $2n-4n_0$ many right eigenvalues of $\tilde{W}_n+\tilde{U}_{(n,n_0)}$ locating in the interval $[\lambda_1(\tilde{W}_n),\lambda_{2n}(\tilde{W}_n)]$, which is further contained in $[-\epsilon,\epsilon]$. Thus, $c(T_n)^{-1}T_n$ has at most $4n_0$ many eigenvalues in $\sigma\left(c(T_n)^{-1}T_n\right)$ lying outside of the interval $[1-\epsilon,1+\epsilon]$,
																			whenever $n>n_0$. 
																			
																			By the discussion above, we see that $\sigma(c(T_n)^{-1}T_n)\subset\mathbb{R}$.
																			{\color{black} Let $\mu_1$ be the smallest eigenvalue in $\sigma(c(T_n)^{-1}T_n)$ with corresponding eigenvector ${\bf z}_1$, that is, $c(T_n)^{-1}T_n{\bf z}_1=\mu_1{\bf z}_1$. Then, for any $n>n_0$, $T_n{\bf z}_1=c(T_n){\bf z}_1\mu_1$ and 
																				\begin{align*}
																					&\mu_1=\frac{{\bf z}_1^*T_n{\bf z}_1}{{\bf z}_1^*c(T_n){\bf z}_1}=\frac{\veccomplexify{z_1}^{*}\matcomplexify{T_n}\veccomplexify{{\bf z}_1}}{\veccomplexify{z_1}^{*}\matcomplexify{c(T_n)}\veccomplexify{{\bf z}_1}} \\
																					&\geq \frac{\lambda_{\min}(\matcomplexify{T_n})||\veccomplexify{{\bf z}_1}||_2}{\lambda_{\max}(\matcomplexify{c(T_n)})||\veccomplexify{{\bf z}_1}||_2}
																					=\frac{\lambda_{\min}(T_n)}{\lambda_{\max}(c(T_n))}\geq \frac{2\check{f}_{\min}}{3\hat{f}_{\max}}>0.
																			\end{align*}}
																			The proof is complete.
																		\end{proof}
																		
																		\section{Proof of Theorem \ref{strangcircmainthm}}\label{proofofcircspcvgmainthm}
																		\begin{proof}
																			It follows from Theorem \ref{strangcircprecedmatspectrumthm} that there exists $n_0>0$ such that $c(T_n)^{-1}T_n$ has at most $4n_0$ eigenvalues in $\sigma(c(T_n)^{-1}T_n)$ lying outside of the interval $[1-\epsilon,1+\epsilon]$ and that $\lambda_{\min}(c(T_n)^{-1}T_n)\geq \frac{2\check{f}_{\min}}{3\hat{f}_{\max}}$. Moreover, from the proof of Theorem \ref{qhpdpcgcvtthm}, we see that $|\sigma(c(T_n)^{-1}T_n)|\leq n$. Thus, the $n$ many right eigenvalues (counting multiplicity) of $c(T_n)^{-1}T_n$ can be ordered as follows:
																			\begin{equation*}
																				\hspace{-3mm}		0<\frac{2\check{f}_{\min}}{3\hat{f}_{\max}}\leq \mu_1\leq \cdots\leq \mu_p\leq 1-\epsilon\leq \mu_{p+1}\leq \cdots\leq \mu_{n-q}\leq 1+\epsilon\leq \mu_{n-q+1}\leq \cdots\leq \mu_{n},
																			\end{equation*}
																			for some positive integers $p,q$ such that $p+q\leq 4n_0$. Then, it follows from Theorem \ref{preparedsuplinearpcgthm} that
																			\begin{align*}
																				||e^{(\ell)}||_{T_n}&\leq 2||e^{(0)}||_{T_n}\left(\frac{\sqrt{(1+\epsilon)/(1-\epsilon)}-1}{\sqrt{(1+\epsilon)/(1-\epsilon)}+1}\right)^{\ell-p-q}\prod\limits_{s=1}^{p}\left(\frac{1+\epsilon-\mu_s}{\mu_s}\right)\\
																				&\leq 2||e^{(0)}||_{T_n}\left(\frac{\sqrt{(1+\epsilon)/(1-\epsilon)}-1}{\sqrt{(1+\epsilon)/(1-\epsilon)}+1}\right)^{\ell-4n_0}\prod\limits_{s=1}^{4n_0}\left(\frac{1+\epsilon-2\check{f}_{\min}/(3\hat{f}_{\max})}{2\check{f}_{\min}/(3\hat{f}_{\max})}\right),
																			\end{align*}
																			for $\ell\geq 4n_0\geq p+q$.
																			Note that
																			\begin{equation*}
																				\frac{\sqrt{(1+\epsilon)/(1-\epsilon)}-1}{\sqrt{(1+\epsilon)/(1-\epsilon)}+1}=\frac{\sqrt{1+\epsilon}-\sqrt{1-\epsilon}}{\sqrt{1+\epsilon}+\sqrt{1-\epsilon}}=\frac{1-\sqrt{1-\epsilon^2}}{\epsilon}\leq \epsilon.
																			\end{equation*}
																			Therefore,
																			\begin{equation*}
																				||e^{(\ell)}||_{T_n}\leq 2||e^{(0)}||_{T_n}\epsilon^{\ell-4n_0}\prod\limits_{s=1}^{4n_0}\left(\frac{1+\epsilon-2\check{f}_{\min}/(3\hat{f}_{\max})}{2\check{f}_{\min}/(3\hat{f}_{\max})}\right),\quad \ell\geq 4n_0.
																			\end{equation*}
																			Taking
																			\begin{equation*}
																				C(\epsilon):=2\left(\frac{1+\epsilon-2\check{f}_{\min}/(3\hat{f}_{\max})}{2\check{f}_{\min}/(3\hat{f}_{\max})}\right)^{4n_0}
																			\end{equation*}
																			leads to the following inequality
																			\begin{equation*}
																				||e^{(\ell)}||_{T_n}\leq C(\epsilon)\epsilon^{\ell-4n_0}||e^{(0)}||_{T_n},\quad \ell\geq 4n_0.
																			\end{equation*}
																			The proof is complete.
																		\end{proof}
																	\end{appendix}

																\end{document}